\documentclass[11pt]{article}
\pdfoutput=1
\usepackage{packages_and_macros}

\title{Reflect-Push Methods Part I: Two Dimensional Techniques}
\author{Nikola Kuzmanovski \and Jamie Radcliffe}
\date{\vspace{-5ex}}

\begin{document}
\maketitle
%auto-ignore
\begin{abstract}
	We determine all maximum weight downsets in the product of two chains, 
	where the weight function is a strictly increasing function of the rank.
	%Furthermore, we precisely determine when a given downset is optimal.
	Many discrete isoperimetric problems can be reduced to the maximum weight downset problem.
	Our results generalize Lindsay's edge-isoperimetric theorem in two dimensions in several directions.
	They also imply and strengthen (in several directions) a result of Ahlswede and Katona concerning graphs with maximal number of adjacent pairs of edges.
	We find all optimal shifted graphs in the Ahlswede-Katona problem.
	Furthermore, the results of Ahlswede-Katona are extended to posets with a rank increasing and rank constant weight function. 
	Our results also strengthen a special case of a recent result by Keough and Radcliffe concerning graphs with the fewest matchings.
	All of these results are achieved by applications of a key lemma that we call the reflect-push method.
	This method is geometric and combinatorial.
	Most of the literature on edge-isoperimetric inequalities focuses on finding a solution, and there are no general methods for finding all possible solutions.
	Our results give a general approach for finding all compressed solutions for the above edge-isoperimetric problems.
	
	By using the Ahlswede-Cai local-global principle, one can conclude that lexicographic solutions are optimal for many cases of higher dimensional isoperimetric problems.
	With this and our two dimensional results we can prove Lindsay's edge-isoperimetric inequality in any dimension.
	Furthermore, our results show that lexicographic solutions are the unique solutions for which compression techniques can be applied in this general setting.
\end{abstract}

%auto-ignore
%auto-ignore
\section{Introduction}\label{intro}

Isoperimetric problems on graphs have been considered for their own sake and also for a variety of applications.
One of the earliest such result is the edge-isoperimetric inequality on hypercubes that was discovered 
by Harper \cite{HARPERLH1964OAoN}, Lindsay \cite{LINDSEYJH1964AoNt} and Bernstein \cite{BernsteinA.J.1967MCAo}, 
and it has been rediscovered many times \cite{ClementsG.F.1971SoLP,HartSergiu1976Anot, Kletiman}.
It says that the maximal number of edges contained in a set of $m$ vertices in a hypercube is achieved by the first $m$ vertices in the lexicographic order.
This problem is equivalent to finding a set of $m$ vertices that minimizes the number of boundary edges.
It turns out that these two classical edge-isoperimetric problems are equivalent for any regular graph.
Viewing hypercubes as the $d$-fold Cartesian products of an edge, one can ask the same question in the product of cliques.
This problem was solved by Lindsay \cite{LINDSEYJH1964AoNt} and lexicographic sets are solutions to it.

The graphs mentioned so far have several nice properties. 
The lexicographic sets that are solutions to the edge-isoperimetric problems are nested.
These graphs are Cartesian products of smaller graphs which enables us to use pushing-down compression techniques.
One can also study discrete isoperimetric problems in graphs that are not Cartesian products or don't have nested solutions.
The earliest known case to the authors is a result by Ahlswede and Katona in \cite{Ahlswede1978GraphsWM}.
Ahlswede and Katona showed that the lex and/or colex graphs have the most pairs of adjacent edges,
among all graphs with a fixed number of vertices and edges.
Equivalently, the lex and/or colex graphs contain the fewest matchings of size $2$.
Keough and Radcliffe extended the results of Ahlswede and Katona to minimizing the total number of matchings and minimizing the number of $k$-matchings.
Again, the lex and colex graph are solutions to these matchings problem.

The problem of Ahlswede and Katona is just an edge-isoperimetric problem in disguise.
Consider the Johnson graph $J(n,d)$, which has vertex set $\binom{[n]}{d}$ and two vertices (i.e. $d$-sets) are adjacent if they have an intersection of size $d-1$.
The Ahlswede and Katona problem is equivalent to the edge-isoperimetric problem on $J(n,2)$.
Note that it doesn't matter which edge-isoperimetric problem is considered because $J(n,d)$ is regular.

The reflect-push method, presented in a general form in Lemma \ref{reflect-push}, 
gives a unified geometric way to handle the edge-isoperimetric problems for all the above mentioned two dimensional graphs.
Furthermore, we can get much stronger results and figure out all compressed optimal downsets, see Theorem \ref{rectangleProof}, Corollary \ref{fullRectangleProof} and Theorem \ref{triangleProof}.
It turns out that every optimal set is an initial segment of the lexicographic or colexicographic order,
or a symmetrization (see Definition \ref{symmetrizationAnyPoset}) of an initial segment of one of these orders.
These results are more general and stated in terms of the maximal weight downset problem and capture all of these previous results as special cases.

We are mainly concerned with two dimensional techniques in this paper.
However, 
after finding all optimal downsets in the product of two chains with 
a rank constant and rank increasing weight function,
one discovers that the only nested solutions are lexicographic or colexicographic.
This immediately gives as corollary that there are lexicographic solutions in higher dimensions from the Ahlswede-Cai local-global principle \cite{AHLSWEDE1997479} (see Theorem \ref{lgbAC}).
Furthermore, this tells us that in the general setting there is a unique order that works with the pushing-down compression approach.
See Section \ref{appliations} for more details.

One might start to wonder if solutions to these types problems are always lexicographic.
This has been studied before and the answer is no.
The first authors to present such a result were Bollob\'as and Leader in \cite{Bollobs1991EdgeisoperimetricII}.
They solved the two classical edge-isoperimetric problems in grids of the form $(P_n)^d$.
Note that the induced edges problem and the boundary edges problem are not equivalent here because grids are not regular graphs.
It turns out that the boundary edges problem does not have nested solutions (for most $n$), while the induced edges problem has nested solutions that are not lexicographic.
The solution to induced edges problem of Bollob\'as and Leader was later generalized by Ahlswede and Bezrukov in \cite{AHLSWEDE199575}, 
for the arbitrary product of trees, but a different method was used along lines of local-global principles.
The solutions are again nested and follow the same type of order as in the grid.
More recently, Bezrukov, Das and Elsässer \cite{BezrukovS.L.2000AEPf} proved that all the powers of the Petersen graph have nested solutions which are not lexicographic.
Using those ideas, Carlson \cite{CarlsonThomasA.2002Tepf} proved that powers of $C_5$ have nested solutions which are again not lexicographic.

There is an important connection between edge-isoperimetric problems and intersection problems.
For this, we introduce a generalization of the Johnson graphs.
The intersection graph $I(n,d,t)$ has vertex set $\binom{[n]}{d}$ and two vertices are adjacent if their intersection is at least $t$.
The maximal clique problem on $I(n,d,1)$ is equivalent to the problem of finding a maximal intersecting family among sets of size $d$.
An answer to this question is provided by the Erd\"os-Ko-Rado Theorem first presented in \cite{EKR}.
The Complete Intersection Theorem of Ahlswede and Khachatrian in \cite{AKComplete, AKPushPull}, determines all maximal cliques in $I(n,d,t)$.
A solution to the edge-isoperimetric problem in $I(n,d,t)$ would be a generalization of the Complete Intersection Theorem.
There has been a lot of interest in this area
\cite{ACCounter, ACHyper, Ahlswede1978GraphsWM, Balogh, AhlswedeBookBey, BEKneser:2003, Brown-KramerJoshua2010Fssw,DGS, Harper1991OnAP, Katz1971RearrangementsO}.
However, very little is known in the case $d>2$.
There are only partial results and nobody has been able to determine an optimal vertex set for each size.
Our results find all shifted optimal sets for $d=2$ and this could give ideas to what type of behavior can occur for $d>3$.

The paper is organized as follows.
Section \ref{notation} introduces notation and terminology.
In Section \ref{rectangles} we find all possible optimal downsets in the product of two chains when the weight function is rank constant and strictly increasing by rank.
We call the products of two chains rectangles. 
Section \ref{rectanglesExact} contains results that completely determines when a certain potentially optimal set (under the conditions from the results in section \ref{rectangles}), 
is optimal.
Section \ref{triangles} of the paper examines posets obtained from the product of two chains, by only keeping the points $(x,y)$ with $x\leq y$.
We call these posets right triangles.
The results obtained about rectangles are used to easily determine all optimal downsets in right triangles.
In Section \ref{appliations} we discuss applications of these results.
We conclude in Section \ref{remarks} with some remarks and open problems.

%auto-ignore
\section{Notation and Terminology}\label{notation}

\begin{dfn}[Natural Numbers]\label{intervals}
	The set $\N$ includes $0$.
	For $n\in \N$ we define $[n]=\{1,\dots, n\}$ and $[n]_0 = \{0,\dots , n-1\}$.
	We also define $[\infty] = \N \setminus \{0\}$.
\end{dfn}

\begin{dfn}[Downsets/Ideals]\label{downsets}
	For a poset $\mathscr{P}$ and set $A\subseteq \mathscr{P}$, 
	we say that $A$ is a \textit{downset with respect to} $\mathscr{P}$ iff $a\in \mathscr{P}$ and $b\in A$ such that $a\leq b$, then $a\in A$.
	A downset is often called an \textit{ideal} in the literature.
\end{dfn}

\begin{dfn}[Weight Functions]\label{weightFunctions}
	Consider a poset $\mathscr{P}$ and a function $\wt: \mathscr{P} \rightarrow \R$.
	We say that $\wt$ is a \textit{weight function on} $\mathscr{P}$ and that $(\mathscr{P}, \wt)$ is a \textit{weighted poset}.
	For $a\in \mathscr{P}$ we say that $\wt(a)$ is \textit{the weight of} $a$.
	For a finite $A\subseteq \mathscr{P}$ we define \textit{the weight of $A$} to be $\wt(A) = \sum_{a\in A} \wt(a)$. 
\end{dfn}

\begin{dfn}[Optimal Downsets]\label{optimalDownsets}
	Suppose that we have a weighted poset $(\mathscr{P}, \wt)$.
	We say that a downset $A\subseteq \mathscr{P}$ is \textit{optimal under} $\wt$ iff
	\begin{align*}
		\wt(A) = \max_{\substack{S\subseteq \mathscr{P}\\ |A|=|S|\\ S \text{ is a downset}}} \wt(S),
	\end{align*}
	and we also call $A$ a \textit{maximum weight downset} in this case.
\end{dfn}

\begin{prb}[Maximum Weight Downset/Ideal Problem]\label{maximumWeightIdealProblem}
	Suppose that we have a weighted poset $(\mathscr{P}, \wt)$.
	For every $m\in [|\mathscr{P}|+1]_0$ find an optimal set $A\subseteq \mathscr{P}$ such that $|A|=m$.
	
	Furthermore, if there exists a sequence $A_0\subseteq A_1\subseteq \cdots \subseteq A_{|\mathscr{P}|}$ such that $|A_i|=i$ and $A_i$ is optimal for all $i$,
	then we say that $\mathscr{P}$ has \textit{nested solutions} under the maximum weight downset/ideal problem.
\end{prb}

\begin{dfn}[Grid Points/Multisets]\label{multisets}
	Suppose that $S$ is a set.
	A \textit{grid point/multiset} on $S$ is a function $f:S \rightarrow \N$.
	If $S$ is finite then we define the weight/size of $f$ by
	\begin{align*}
		|f| = \sum_{x\in S} f(x),
	\end{align*}
	and we say that $f$ has dimension $|S|$.
\end{dfn}

\begin{dfn}[Grids/Restricted Multisets]\label{setOfMultisets}
	Suppose that we have a set $S= \{a_1,\dots, a_d\}$ and consider $\ell_1,\dots, \ell_d\in \N \cup \{\infty\}$,
	with $d\geq 1$.
	We define the \textit{grid/set of all multisets} $S$ and \textit{lengths/repetition limits} $(\ell_1,\dots, \ell_d)$ by
	\begin{align*}
		M_S(\ell_1,\dots, \ell_d) = \{f:S\rightarrow \N \bigm | f(a_1)< \ell_1,\dots, f(a_d)< \ell_d\}.
	\end{align*}
\end{dfn}

\begin{dfn}[Grid Lattice/Lattice of Multisets]\label{latticeOfMultisets}
	Suppose that $f,g$ are multisets on $S$ with $|S|=d\geq 1$ and consider $\ell_1,\dots, \ell_d\in \N \cup \{\infty\}$.
	We say that $f$ is a  \textit{submultiset} of $g$ and write $f\subseteq g$ if and only if for all $x\in S$ we have $f(x) \leq g(x)$.
	Then $\subseteq$ is a partial order on $M_S(\ell_1,\dots, \ell_d)$.
	The \textit{lattice of multisets of dimension} $d$ \textit{on} $S$ \textit{and lengths} $(\ell_1,\dots, \ell_d)$ is the poset 
	\begin{align*}
		\mathscr{M}_S(\ell_1,\dots, \ell_d) = (M_S(\ell_1,\dots, \ell_d), \subseteq).
	\end{align*}
\end{dfn}

\begin{dfn}[Shadows]\label{downShadowsDef}
	Suppose that $\mathscr{P}$ is a poset and $a,b\in \mathscr{P}$.
	We say that $a$ \textit{covers} $b$ if and only if $b < a$ and there is no $x\in \mathscr{P}$ such that $b < x < a$.
	Also, we say that $b$ is a \textit{lower shadow point} of $a$ if $a$ covers $b$, and we say that $b$ is an \textit{upper shadow point} of $a$ if $b$ covers $a$. 
	By $\dSdw_{\mathscr{P}}(a)$ and $\uSdw_{\mathscr{P}}(a)$ we denote the \textit{set of all lower shadow points of} $a$ in $\mathscr{P}$
	and the \textit{set of all upper shadow points of} $a$ in $\mathscr{P}$ respectively.
	The subscript $\mathscr{P}$ will often be omitted.
\end{dfn}

\begin{rem}
	All tosets (totally ordered sets) considered in this paper will be isomorphic to $[n]$ with the standard order, where $n\in \N \cup \{\infty\}$.
\end{rem}

\begin{dfn}[Indices, Elements and Intervals]\label{indicesElementsAndIntervals}
	Let $\mathscr{T}$ be a toset.
	Then there is a bijective function $\sigma: \mathscr{T} \rightarrow [n]$ such that for any $a,b\in \mathscr{T}$ we have $a\leq b$ iff $\sigma(a) \leq \sigma(b)$.
	For any $x,y\in \mathscr{T}$ with $x \leq y$ and any $p,q\in [n]$ with $p\leq q$ we define:
	\begin{enumerate}
		\item The \textit{index of} $x$, $\mathscr{T}(x) = \sigma(x)$.
		\item The \textit{element of} $p$, $\mathscr{T}^{-1}(p) = \sigma^{-1}(p)$.
		\item The \textit{interval between} $x$ \textit{and} $y$,
		\begin{align*}
			\mathscr{T}[x,y] &= \{a\in \mathscr{T} \bigm | x\leq a \leq y\}.
		\end{align*} 
		\item The \textit{interval between} $p$ \textit{and} $q$,
		\begin{align*}
			\mathscr{T}^{-1}[p,q] &= \mathscr{T}[\mathscr{T}^{-1}(p), \mathscr{T}^{-1}(q)].
		\end{align*}
	\end{enumerate}
	We call $\mathscr{T}^{-1}[q]$ the \textit{initial segment of size $q$ in} $\mathscr{T}$.
	We sometimes just use the total order of $\mathscr{T}$, which lets say we call $\mathcal{O}$,
	and we call $\mathscr{T}^{-1}[q]$ the \textit{initial segment of size $q$ of} $\mathcal{O}$.
\end{dfn}

\begin{prp}[Decomposition of Multiset Lattices]\label{decompositionOfMultisetLatices}
	Suppose that $d\in \N$ with $d\geq 1$
	Consider tosets $\mathscr{T}_1,\dots, \mathscr{T}_d$ and put $\mathscr{T}=\mathscr{T}_1\times \cdots \times \mathscr{T}_d$.
	One has,
	\begin{align*}
		\mathscr{T} \iso \mathscr{M}_{[d]}(|\mathscr{T}_1|,\dots, |\mathscr{T}_d|),
	\end{align*}
	where the isomorphism sends $(x_1,\dots, x_d)\in \mathscr{T}$ to $(\mathscr{T}(x_1)-1,\dots, \mathscr{T}(x_d)-1)$.
	We call this isomorphism the \textit{standard multiset isomorphism of} $\mathscr{T}$ and denote it by $\mu_{\mathscr{T}}$.
\end{prp}

\begin{dfn}[Standard Weights on Multisets]\label{standardWeightsOnMultisets}
	Suppose that $d\in \N$ with $d\geq 1$
	Consider tosets $\mathscr{T}_1,\dots, \mathscr{T}_d$ and put $\mathscr{T}=\mathscr{T}_1\times \cdots \times \mathscr{T}_d$.
	The \textit{standard weight function} on $\mathscr{T}$ is the weight function that for each $x\in \mathscr{T}$ assigns the weight $|\mu_{\mathscr{T}}(x)|$.
	We will often call the number $|\mu_{\mathscr{T}}(x)|$, the \textit{rank} of $x$.
\end{dfn}

\begin{dfn}[Cube Diagrams]\label{cubeDiagrams}
	Suppose that $\mathscr{T}_1,\dots, \mathscr{T}_d$ are tosets and put $\mathscr{T}=\mathscr{T}_1\times \dots \times \mathscr{T}_d$, 
	where $d\in \N$ is nonzero.
	Take $x=(x_1,\dots, x_d)\in \mathscr{T}$.
	Then we can consider $(\mathscr{T}_1(x_1) - 1,\dots, \mathscr{T}_d(x_d) - 1)\in \R^d$.
	For the point $(\mathscr{T}_1(x_1) - 1,\dots, \mathscr{T}_d(x_d) - 1)$ we consider the \textit{upwards cube of length $1$ in} $\R^d$,
	\begin{align*}
		\prod_{i=1}^d [\mathscr{T}_i(x_i),\mathscr{T}_i(x_i)+1]
	\end{align*}
	Thus, for each $x=(x_1,\dots, x_d)\in \mathscr{T}$ there is a unique upwards cube in $\R^d$.
	For a set $A\subseteq \mathscr{T}$ we call the collection of all upwards cubes from elements in $A$ the \textit{cube diagram of} $A$ in $\mathscr{T}$.
\end{dfn}

\begin{figure}
	\centering
	\begin{subfigure}[t]{0.23\textwidth}
		\includegraphics[width=\textwidth]{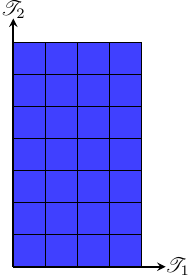}
		\caption{$\emptyset$.}
	\end{subfigure}
	\hfill
	\begin{subfigure}[t]{0.23\textwidth}
		\includegraphics[width=\textwidth]{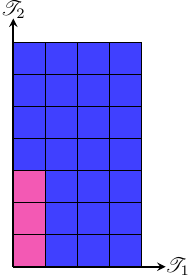}
		\caption{$\{(0,0),(0,1),(0,2)\}$.}
	\end{subfigure}
	\hfill
	\begin{subfigure}[t]{0.23\textwidth}
		\includegraphics[width=\textwidth]{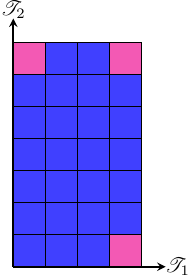}
		\caption{$\{(0,6),(3,6),(3,0)\}$.}
	\end{subfigure}
	\hfill
	\begin{subfigure}[t]{0.23\textwidth}
		\includegraphics[width=\textwidth]{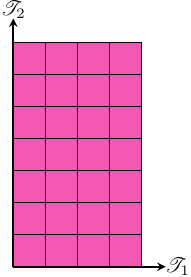}
		\caption{$\mathscr{M}_{[d]}(4,7)$.}
	\end{subfigure}

	\caption{Cube diagrams for some subsets in $\mathscr{M}_{[d]}(4,7)$.}
	\label{cubeDiagrams_4x7}
\end{figure}

\begin{figure}
	\centering
	\begin{subfigure}[t]{0.75\textwidth}
		\includegraphics[width=\textwidth]{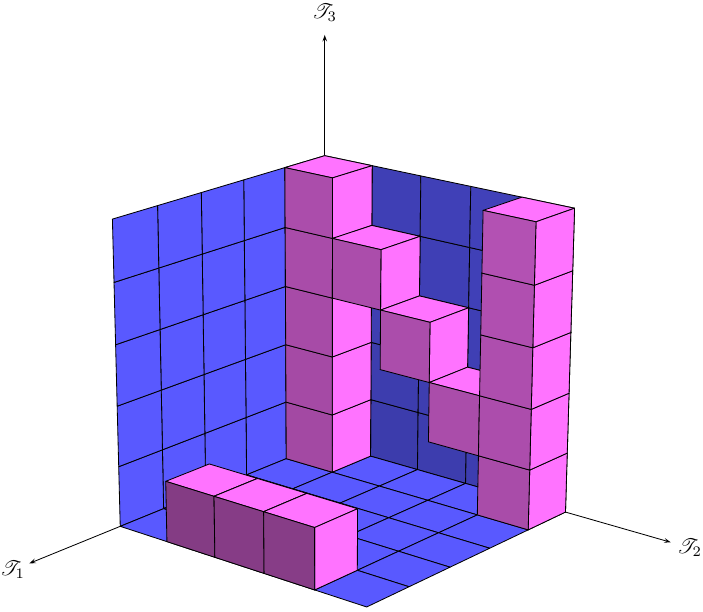}
	\end{subfigure}
	
	\begin{subfigure}[t]{0.46\textwidth}
		\includegraphics[width=\textwidth]{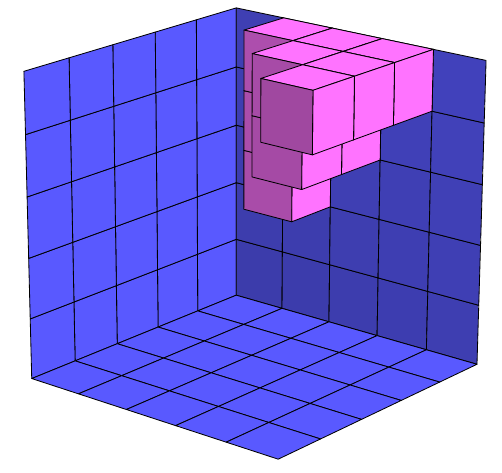}
	\end{subfigure}
	\hfill
	\begin{subfigure}[t]{0.46\textwidth}
		\includegraphics[width=\textwidth]{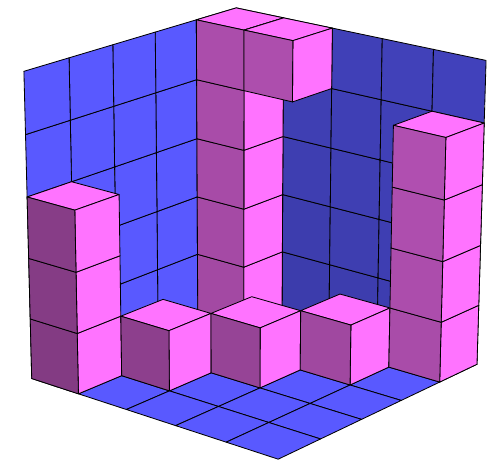}
	\end{subfigure}
	
	\caption{Cube diagrams of some sets in $\mathscr{M}_{[3]}(5,5,5)$.}\label{cubeDiagrams_5x5x5}
\end{figure}

\begin{exm}
	Some cube diagrams of sets in $\mathscr{M}_{[2]}(4,7)$ are shown in Figure \ref{cubeDiagrams_4x7}.
	Some cube diagrams of sets in $\mathscr{M}_{[3]}(5,5,5)$ are shown in Figure \ref{cubeDiagrams_5x5x5}.
\end{exm}

\begin{dfn}[Lexicographic Order]\label{lexicographicOrder}
	Suppose that we have tosets $\mathscr{T}_1, \dots , \mathscr{T}_d$ for $d\in \N$ with $d\geq 1$.
	Consider $\mathscr{T} = \mathscr{T}_1\times \cdots \times \mathscr{T}_d$ and $x=(x_1,\dots, x_d),y=(y_1,\dots, y_d)\in \mathscr{T}$.
	We define the \textit{lexicographic order} on $\mathscr{T}$, $\mathcal{L}_{\mathscr{T}}$ to be a total order on $\mathscr{T}$,
	such that $x<_{\mathcal{L}_{\mathscr{T}}} y$ 
	iff for some $i\in \{1,\dots, d-1\}$ we have $x_1=y_1,\dots, x_i=y_i$ and $x_{i+1} <_{\mathscr{T}_{i+1}} y_{i+1}$.
	We abuse notation and treat $\mathscr{T}$ as its ground set, and define the toset $\mathscr{T}_{\mathcal{L}}=(\mathscr{T}, \mathcal{L}_{\mathscr{T}})$.
\end{dfn}

\begin{figure}
	\centering
	\begin{subfigure}[t]{0.18\textwidth}
		\includegraphics[width=\textwidth]{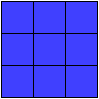}
	\end{subfigure}
	\hfill
	\begin{subfigure}[t]{0.18\textwidth}
		\includegraphics[width=\textwidth]{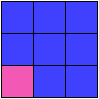}
	\end{subfigure}
	\hfill
	\begin{subfigure}[t]{0.18\textwidth}
		\includegraphics[width=\textwidth]{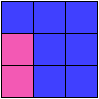}
	\end{subfigure}
	\hfill
	\begin{subfigure}[t]{0.18\textwidth}
		\includegraphics[width=\textwidth]{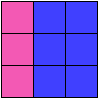}
	\end{subfigure}
	\hfill
	\begin{subfigure}[t]{0.18\textwidth}
		\includegraphics[width=\textwidth]{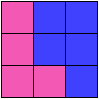}
	\end{subfigure}
	%%%%%%%%%%%%%%%%%%%%%%%%%%%%%%%%%%%%%%%%%%%%%%%%%%%%%%%%%%%%%%%%%%%%%%%%%%%%%%%%%%%%%%%%%%%%%%%%%%%%%%%%%%%%%%%%
	%%%%%%%%%%%%%%%%%%%%%%%%%%%%%%%%%%%%%%%%%%%%%%%%%%%%%%%%%%%%%%%%%%%%%%%%%%%%%%%%%%%%%%%%%%%%%%%%%%%%%%%%%%%%%%%%
	%%%%%%%%%%%%%%%%%%%%%%%%%%%%%%%%%%%%%%%%%%%%%%%%%%%%%%%%%%%%%%%%%%%%%%%%%%%%%%%%%%%%%%%%%%%%%%%%%%%%%%%%%%%%%%%%
	
	\vspace{0.5cm}
	
	\begin{subfigure}[t]{0.18\textwidth}
		\includegraphics[width=\textwidth]{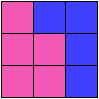}
	\end{subfigure}
	\hfill
	\begin{subfigure}[t]{0.18\textwidth}
		\includegraphics[width=\textwidth]{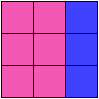}
	\end{subfigure}
	\hfill
	\begin{subfigure}[t]{0.18\textwidth}
		\includegraphics[width=\textwidth]{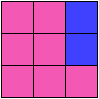}
	\end{subfigure}
	\hfill
	\begin{subfigure}[t]{0.18\textwidth}
		\includegraphics[width=\textwidth]{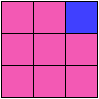}
	\end{subfigure}
	\hfill
	\begin{subfigure}[t]{0.18\textwidth}
		\includegraphics[width=\textwidth]{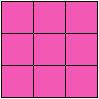}
	\end{subfigure}

	\caption{Lexicographic order in $\mathscr{M}_{[2]}(3,3)$.}\label{Lex_3x3}
\end{figure}

\begin{exm}
	A visualization of lexicographic order in $\mathscr{M}_{[2]}(3,3)$ can be seen in Figure \ref{Lex_3x3}.
\end{exm}

\begin{dfn}[Symmetric Group]\label{symmetricGroup}
	For $d\in \N$ with $d\geq 1$, by $\mathfrak{S}_d$ we denote the \textit{set of all permutations on} $[d]$.
\end{dfn}

\begin{dfn}[Domination Order]\label{dominationOrder}
	Suppose that $d\in \N$ with $d\geq 1$ and consider tosets $\mathscr{T}_1,\dots, \mathscr{T}_d$.
	Also, put $\mathscr{T}= \mathscr{T}_1\times \cdots \times \mathscr{T}_d$.
	For any $\pi \in \mathfrak{S}_d$ we define $\mathcal{D}_{\pi}$, the {\em domination order induced from $\pi$ on} $\mathscr{T}$,
	such that for any $x=(x_1,\dots, x_d),y=(y_1,\dots, y_d)\in \mathscr{T}$ we have
	\begin{align*}
		x \leq_{\mathcal{D}_\pi} y \text{ iff } (x_{\pi(1)},\dots, x_{\pi(d)}) \leq_{\mathscr{T}_{\mathcal{L}}} (y_{\pi(1)},\dots, y_{\pi(d)}).
	\end{align*}
	Then we define the toset $\mathscr{T}_{\pi} = (\mathscr{T}, \mathcal{D}_{\pi})$.
\end{dfn}

\begin{rem}
	The lexicographic order on a product of tosets is a domination order.
	It is the domination order induced by the identity permutation.
	Using the notation developed so far, the statement in the previous sentence becomes $\mathscr{T}_{\mathcal{L}} = \mathscr{T}_{\id_{[d]}}$.
\end{rem}

\begin{dfn}[Colexicographic Order]\label{colexicographicOrder}
	Suppose that $d\in \N$ with $d\geq 1$ and consider tosets $\mathscr{T}_1,\dots, \mathscr{T}_d$
	Also, put $\mathscr{T}= \mathscr{T}_1\times \cdots \times \mathscr{T}_d$.
	Let $\pi \in \mathfrak{S}_d$ be the permutation such that $\pi(i) = d-i+1$ for all $i\in \{1,\dots, d\}$.
	The \textit{colexicographic order} $\mathcal{C}_{\mathscr{T}}$ on $\mathscr{T}$ is defined to be the domination order $\mathcal{D}_{\pi}$.
	We also write $\mathscr{T}_{\mathcal{C}}$ for $\mathscr{T}_\pi$.
	Another way to view the colexicographic order is to take any $x=(x_1,\dots, x_d),y=(y_1,\dots, y_d)\in \mathscr{T}$ and define $x<_{\mathcal{C}_{\mathscr{T}}} y$ 
	iff for some $i\in \{2,\dots, d+1\}$ we have $x_d=y_d,\dots, x_i=y_i$ and $x_{i-1} <_{\mathscr{T}_{i+1}} y_{i-1}$.
\end{dfn}

\begin{figure}
	\centering
	\begin{subfigure}[t]{0.19\textwidth}
		\includegraphics[width=\textwidth]{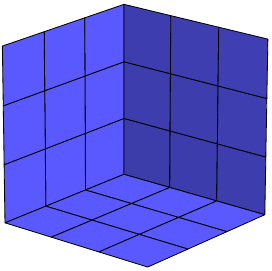}
	\end{subfigure}
	\hfill
	\begin{subfigure}[t]{0.19\textwidth}
		\includegraphics[width=\textwidth]{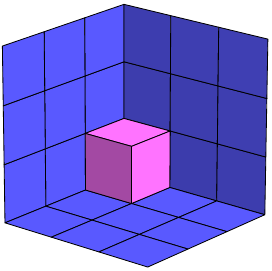}
	\end{subfigure}
	\hfill
	\begin{subfigure}[t]{0.19\textwidth}
		\includegraphics[width=\textwidth]{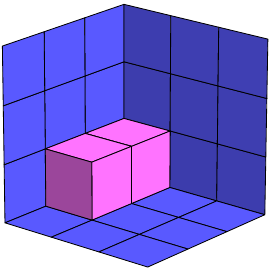}
	\end{subfigure}
	\hfill
	\begin{subfigure}[t]{0.19\textwidth}
		\includegraphics[width=\textwidth]{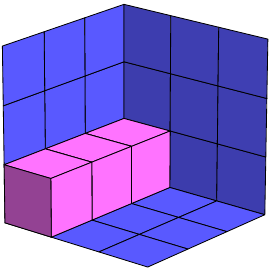}
	\end{subfigure}
	\hfill
	\begin{subfigure}[t]{0.19\textwidth}
		\includegraphics[width=\textwidth]{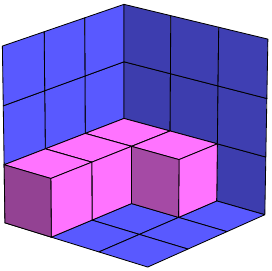}
	\end{subfigure}
	%%%%%%%%%%%%%%%%%%%%%%%%%%%%%%%%%%%%%%%%%%%%%%%%%%%%%%%%%%%%%%%%%%%%%%%%%%%%%%%%%%%%%%%%%%%%%%%%%%%%%%%%%%%%%%%%
	%%%%%%%%%%%%%%%%%%%%%%%%%%%%%%%%%%%%%%%%%%%%%%%%%%%%%%%%%%%%%%%%%%%%%%%%%%%%%%%%%%%%%%%%%%%%%%%%%%%%%%%%%%%%%%%%
	%%%%%%%%%%%%%%%%%%%%%%%%%%%%%%%%%%%%%%%%%%%%%%%%%%%%%%%%%%%%%%%%%%%%%%%%%%%%%%%%%%%%%%%%%%%%%%%%%%%%%%%%%%%%%%%%
	
	\vspace{0.1cm}
	
	\begin{subfigure}[t]{0.19\textwidth}
		\includegraphics[width=\textwidth]{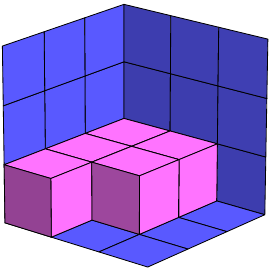}
	\end{subfigure}
	\hfill
	\begin{subfigure}[t]{0.19\textwidth}
		\includegraphics[width=\textwidth]{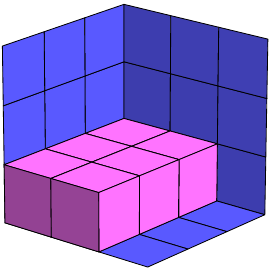}
	\end{subfigure}
	\hfill
	\begin{subfigure}[t]{0.19\textwidth}
		\includegraphics[width=\textwidth]{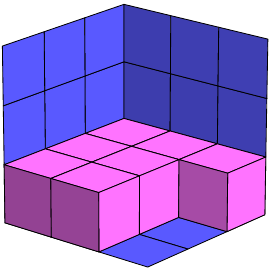}
	\end{subfigure}
	\hfill
	\begin{subfigure}[b]{0.19\textwidth}
		\centering
		\raisebox{0.5\hsize}{\includegraphics[width=0.5\textwidth]{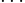}}
	\end{subfigure}
	\hfill
	\begin{subfigure}[t]{0.19\textwidth}
		\includegraphics[width=\textwidth]{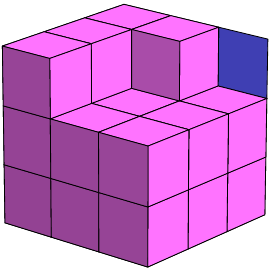}
	\end{subfigure}
	%%%%%%%%%%%%%%%%%%%%%%%%%%%%%%%%%%%%%%%%%%%%%%%%%%%%%%%%%%%%%%%%%%%%%%%%%%%%%%%%%%%%%%%%%%%%%%%%%%%%%%%%%%%%%%%%
	%%%%%%%%%%%%%%%%%%%%%%%%%%%%%%%%%%%%%%%%%%%%%%%%%%%%%%%%%%%%%%%%%%%%%%%%%%%%%%%%%%%%%%%%%%%%%%%%%%%%%%%%%%%%%%%%
	%%%%%%%%%%%%%%%%%%%%%%%%%%%%%%%%%%%%%%%%%%%%%%%%%%%%%%%%%%%%%%%%%%%%%%%%%%%%%%%%%%%%%%%%%%%%%%%%%%%%%%%%%%%%%%%%
	
	\vspace{0.1cm}
	
	\begin{subfigure}[t]{0.19\textwidth}
		\includegraphics[width=\textwidth]{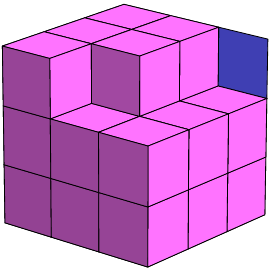}
	\end{subfigure}
	\hfill
	\begin{subfigure}[t]{0.19\textwidth}
		\includegraphics[width=\textwidth]{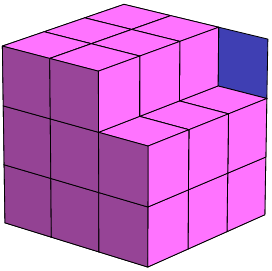}
	\end{subfigure}
	\hfill
	\begin{subfigure}[t]{0.19\textwidth}
		\includegraphics[width=\textwidth]{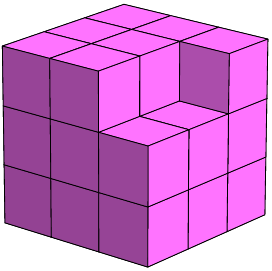}
	\end{subfigure}
	\hfill
	\begin{subfigure}[t]{0.19\textwidth}
		\includegraphics[width=\textwidth]{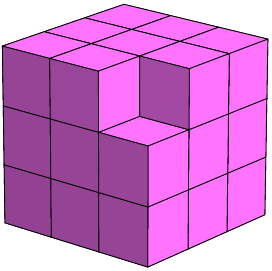}
	\end{subfigure}
	\hfill
	\begin{subfigure}[t]{0.19\textwidth}
		\includegraphics[width=\textwidth]{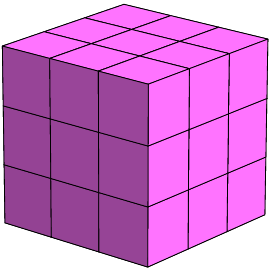}
	\end{subfigure}
	
	\caption{Colexicographic order in $\mathscr{M}_{[3]}(3,3,3)$.}
	\label{Colex_3x3x3}
\end{figure}

\begin{exm}
	A visualization of colexicographic order in $\mathscr{M}_{[3]}(3,3,3)$ can be seen in Figure \ref{Colex_3x3x3}.
\end{exm}

%auto-ignore
\section{All Optimal Downsets in Rectangles}\label{rectangles}

\begin{dfn}[Subproducts]\label{subproducts}
	Suppose that $d\geq 1$, 
	we have tosets $\mathscr{T}_1,\dots, \mathscr{T}_d$, 
	and put $\mathscr{T}=\mathscr{T}_1\times \cdots \times \mathscr{T}_d$.
	For a set of coordinates $S=\{p_1 < \dots < p_k \}\subseteq [d]$ we define the \textit{subproduct of} $\mathscr{T}$ \textit{under} $S$ by
	\begin{align*}
		\mathscr{T}_S = \mathscr{T}_{p_1} \times \cdots \times \mathscr{T}_{p_k},
	\end{align*} 
	where $\mathscr{T}_{\emptyset} = \emptyset$.
	We say that $\mathscr{T}_S$ \textit{has dimension} $k$.
	Let $\overline{S} = [d]\setminus S$ and we can write $\overline{S}=\{q_1 < \cdots < q_{d-k}\}$.
	For $x=(x_{q_1},\dots, x_{q_{d-k}})\in \mathscr{T}_{\overline{S}}$ we define the \textit{subproduct at} $x$ \textit{under} $S$ \textit{of} $\mathscr{T}$ by
	\begin{align*}
		\mathscr{T}_S(x) = \{(y_1,\dots, y_d)\in \mathscr{T} \bigm | y_a=x_a \text{ for all } a\in \overline{S}\}.
	\end{align*}
\end{dfn}

\begin{figure}
	\centering
	\begin{subfigure}[t]{0.3\textwidth}
		\includegraphics[width=\textwidth]{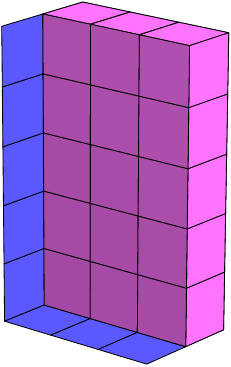}
		\caption{$(\mathscr{M}_{[3]}(2,3,5))_{\{2,3\}}(0)$.}
	\end{subfigure}
	\hfill
	\begin{subfigure}[t]{0.3\textwidth}
		\includegraphics[width=\textwidth]{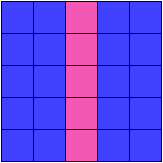}
		\caption{$(\mathscr{M}_{[2]}(5,5))_{\{2\}}(2)$.}
	\end{subfigure}
	\hfill
	\begin{subfigure}[t]{0.3\textwidth}
		\includegraphics[width=\textwidth]{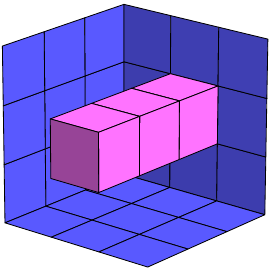}
		\caption{$(\mathscr{M}_{[3]}(3,3,3))_{\{1\}}((1,1))$.}
	\end{subfigure}
	
	\caption{Some subproducts of multiset lattices.}
	\label{subproductsExample}
\end{figure}

\begin{exm}
	Some example of subproducts can be seen in Figure \ref{subproductsExample}.
\end{exm}

\begin{dfn}[Reflections and Symmetrization]\label{reflectionsAndSymmetrization}
	Let $\mathscr{M} = \mathscr{M}_{[d]}(\ell_1,\dots, \ell_d)$ be a finite multiset lattice and consider coordinates $c_1,c_2\in [d]$ such that $\ell_{c_1} = \ell_{c_2}$.
	For $f\in \mathscr{M}$ we define the \textit{reflection of} $f$ \textit{about} $\{c_1,c_2\}$ to be $f_{\{c_1,c_2\}} \in \mathscr{M}$,
	such that $f_{\{c_1,c_2\}}(c_1) = f(c_2)$, $f_{\{c_1,c_2\}}(c_2) = f(c_1)$, 
	and for all $c\in [d]\setminus \{c_1,c_2\}$ we have $f_{\{c_1,c_2\}}(c) = f(c)$.
	Then for $f\in \mathscr{M}$ we define the \textit{symmetrization of} $f$ \textit{about} $(c_1,c_2)$ by $f_{(c_1,c_2)} \in \mathscr{M}$,
	such that $f_{(c_1,c_2)} = f_{\{c_1,c_2\}}$ if $f(c_1) < f(c_2)$, and $f_{(c_1,c_2)} = f$ if $f(c_1)\geq f(c_2)$. 
	Finally, for a set $A\subseteq \mathscr{M}$ the \textit{symmetrization of} $A$ \textit{about} $(c_1,c_2)$ \textit{in} $\mathscr{M}$ is
	\begin{align*}
		\Sym_{\mathscr{M}}(A,c_1,c_2) = \{f\in \mathscr{M}\bigm | f\in A \text{ and } f_{\{c_1,c_2\}}\in A \} 
										\cup \{f_{(c_1,c_2)}\in \mathscr{M} \bigm | f\in A\}.
	\end{align*}
\end{dfn}

\begin{exm}
	On Figure \ref{reflectionExample} one can see reflections and symmetrizations of single elements.
	On Figure \ref{symmetrizationExample} one can observe the symmetrization of an entire set.
\end{exm}

\begin{figure}
	\centering
	\begin{subfigure}[t]{0.3\textwidth}
		\includegraphics[width=\textwidth]{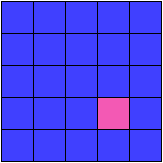}
		\caption{$(3,1) \in \mathscr{M}_{[2]}(5,5)$.}
	\end{subfigure}
	\begin{subfigure}[b]{0.2\textwidth}
		\centering
		\raisebox{0.75\hsize}{\includegraphics[width=\textwidth]{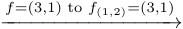}}
	\end{subfigure}
	\begin{subfigure}[t]{0.3\textwidth}
		\includegraphics[width=\textwidth]{section2/2D_Pictures/cube_diagram_5x5_reflection_1.pdf}
		\caption{$(3,1) \in \mathscr{M}_{[2]}(5,5)$.}
	\end{subfigure}
	
	\vspace{0.5cm}
	
	\begin{subfigure}[t]{0.3\textwidth}
		\includegraphics[width=\textwidth]{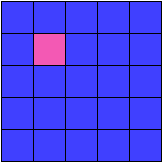}
		\caption{$(1,3) \in \mathscr{M}_{[2]}(5,5)$.}
	\end{subfigure}
	\begin{subfigure}[c]{0.2\textwidth}
		\centering
		\raisebox{0.75\hsize}[0pt][0pt]{\includegraphics[width=\textwidth]{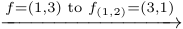}}
	\end{subfigure}
	\begin{subfigure}[t]{0.3\textwidth}
		\includegraphics[width=\textwidth]{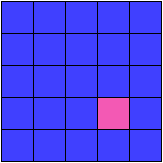}
		\caption{$(3,1) \in \mathscr{M}_{[2]}(5,5)$.}
	\end{subfigure}

	\caption{Points in $\mathscr{M}_{[2]}(5,5)$ and their symmetrizations about $(1,2)$.}
	\label{reflectionExample}
\end{figure}

\begin{figure}
	\centering
	\begin{subfigure}[t]{0.3\textwidth}
		\includegraphics[width=\textwidth]{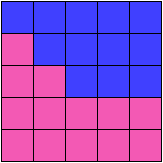}
		\caption{$A$.}
	\end{subfigure}
	\begin{subfigure}[b]{0.2\textwidth}
		\centering
		\raisebox{0.75\hsize}{\includegraphics[width=\textwidth]{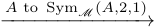}}
	\end{subfigure}
	\begin{subfigure}[t]{0.3\textwidth}
		\includegraphics[width=\textwidth]{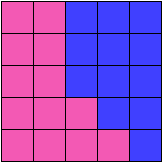}
		\caption{$\Sym_{\mathscr{M}}(A,2,1)$.}
	\end{subfigure}
	
	\caption{The symmetrization of a set $A$ about $(2,1)$ in $\mathscr{M}=\mathscr{M}_{[2]}(5,5)$.}
	\label{symmetrizationExample}
\end{figure}

\begin{dfn}[Symmetrization in Arbitrary Posets]\label{symmetrizationAnyPoset}
	Suppose that we have a poset $\mathscr{P}$.
	Furthermore, suppose that we have a subposet $\mathscr{Q} \subseteq \mathscr{P}$,
	that is isomorphic to a finite multiset lattice $\mathscr{M} = \mathscr{M}_{[d]}(\ell_1,\dots, \ell_d)$.
	For a set $A\subseteq \mathscr{P}$ and coordinates $c_1,c_2\in [d]$ with $\ell_{c_1} = \ell_{c_2}$,
	we define \textit{the symmetrization of $A$ about $(c_1,c_2)$ and with respect to} $\mathscr{Q}$ as
	\begin{align*}
		\Sym_{\mathscr{P}}(\mathscr{Q},A,c_1,c_2) = (A\setminus \mathscr{Q})\cup \Sym_{\mathscr{M}}(A\cap \mathscr{Q}, c_1,c_2).
	\end{align*}
\end{dfn}

\begin{dfn}[Packed Posets]
	Suppose that we have a finite multiset lattice $\mathscr{T}=\mathscr{T}_1\times \cdots \times \mathscr{T}_d$
	and take a subset of coordinates $S=\{x_1<\dots <x_k\}\subseteq [d]$.
	Consider a subposet $\mathscr{Q} \subseteq \mathscr{T}$ that is isomorphic to a multiset lattice $\mathscr{M}=\mathscr{M}_{S}(\ell_1,\dots, \ell_d)$.
	We say that $\mathscr{Q}$ is \textit{packed} if $\mathscr{Q} =\mathscr{T}_{x_1}^{-1}[a_{1}, b_{1}]\times \cdots \times \mathscr{T}_{x_k}^{-1}[a_k, b_k]$,
	for some $a_i,b_i\in [|\mathscr{T}_{x_i}|]_0$ with $a_i\leq b_i$.
\end{dfn}

\begin{exm}
	On Figure \ref{subposets} we can see an example of packed poset and an example of a poset that is not packed.
	On Figure \ref{symmetrizationPackedExample} we can see a symmetrization with respect to a packed poset.
\end{exm}

\begin{figure}
	\centering
	\begin{subfigure}[t]{0.3\textwidth}
		\includegraphics[width=\textwidth]{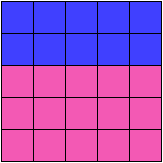}
		\caption{A packed poset.}
	\end{subfigure}
	\hspace{0.5cm}
	\begin{subfigure}[t]{0.3\textwidth}
		\includegraphics[width=\textwidth]{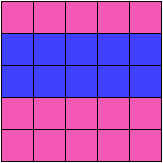}
		\caption{A poset that is not packed.}
	\end{subfigure}
	\caption{Isomorphic copies of $\mathscr{M}_{[2]}(5,3)$ inside $\mathscr{M}_{[2]}(5,5)$.}
\label{subposets}
\end{figure}

\begin{figure}
	\centering
	\begin{subfigure}[t]{0.38\textwidth}
		\includegraphics[width=\textwidth]{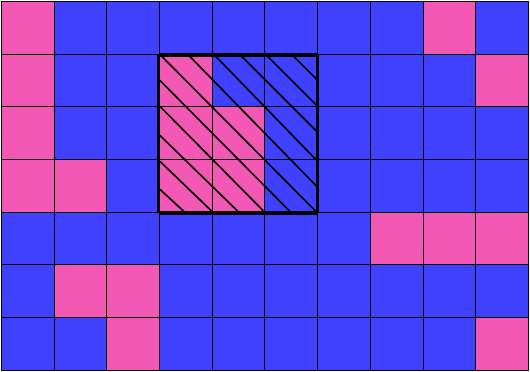}
		\caption{The set $A$ and poset $\mathscr{Q}$ (shaded).}
	\end{subfigure}
	\hfill
	\begin{subfigure}[b]{0.19\textwidth}
		\raisebox{0.7\hsize}{\includegraphics[width=\textwidth]{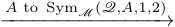}}
	\end{subfigure}
	\hfill
	\begin{subfigure}[t]{0.38\textwidth}
		\includegraphics[width=\textwidth]{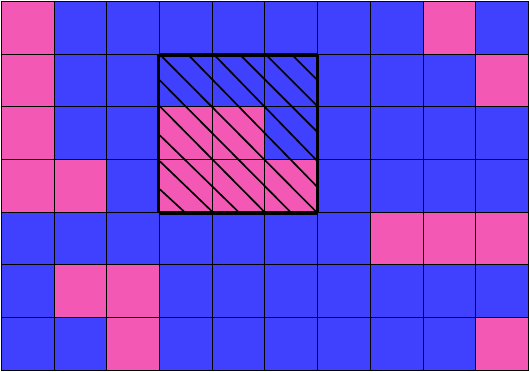}
		\caption{The set $\Sym_{\mathscr{M}}(\mathscr{Q},A,1,2)$.}
	\end{subfigure}

	\caption{The symmetrization of a set $A$ inside $\mathscr{M}=\mathscr{M}_{[2]}(10,7)$ about $(1,2)$ with respect to the packed subposet $\mathscr{Q} \iso \mathscr{M}_{[2]}(3,3)$.}
	\label{symmetrizationPackedExample}
\end{figure}

\begin{dfn}[Rank Increasing Weight Functions]\label{rankIncreasing}
	Suppose that $\mathscr{M}$ is a multiset lattice and $\wt$ is a weight function on $\mathscr{M}$.
	We say that $\wt$ is \textit{rank increasing} if whenever $x,y, \in \mathscr{M}$ 
	with $x$ having smaller rank than $y$ ($|\mu_{\mathscr{M}}(x)| < |\mu_{\mathscr{M}}(y)|$),
	then $\wt(x) < \wt(y)$.
\end{dfn}

\begin{dfn}[Rank Constant Weight Functions]\label{rankConstant}
	Suppose that $\mathscr{M}$ is a multiset lattice and $\wt$ is a weight function on $\mathscr{M}$.
	We say that $\wt$ is \textit{rank constant} if whenever $x,y, \in \mathscr{M}$ 
	with $x$ having the same rank as $y$ ($|\mu_{\mathscr{M}}(x)| = |\mu_{\mathscr{M}}(y)|$),
	then $\wt(x) = \wt(y)$.
\end{dfn}

\begin{rem}
	The standard weight function on any multiset lattice is rank increasing and rank constant.
\end{rem}

%\begin{exm}
%	A visualization of a rank increasing and rank symmetric weight function can be observed in Figure \ref{rankIncreasingRankSymExample}.
%\end{exm}

%\begin{figure}
%	\centering
%	\includegraphics[width=0.5\textwidth]{section2/2D_Pictures/cube_diagram_7x10_weights.pdf}
%	\caption{A rank increasing and rank symmetric weight function on $\mathscr{M}_{[2]}(10,7)$.}
%	\label{rankIncreasingRankSymExample}
%\end{figure}

\begin{lem}[Reflect-Push Method]\label{reflect-push}
	Suppose that $\mathscr{M} = \mathscr{M}_{[d]}(\ell_1,\dots, \ell_d)$ is a finite multiset lattice with a rank increasing and rank constant weight function.
	Furthermore, suppose that the following hold:
	\begin{enumerate}
		\item We have a downset $A\subseteq \mathscr{M}$.
		\item There is a packed poset $\mathscr{Q} \subseteq \mathscr{M}$.
		\item We can find a set $O \subseteq A\cap \mathscr{Q}$ such that $A\setminus O$ is a downset.
		\item There exist coordinates $c_1,c_2\in [d]$ for which we define the set 
		\begin{align*}
			R = \{f_{\{c_1,c_2\}} \bigm | f\in O\},
		\end{align*}
		where the reflections are happening in $\mathscr{Q}$ and the lengths in the directions of $c_1$ and $c_2$ are the same inside $\mathscr{Q}$.
		\item We have a set $P\subseteq \mathscr{M}\setminus A$ such that $(A\setminus O) \cup P$ is a downset.
		\item There is a bijective function $\sigma: R \rightarrow P$ such that for all $f\in R$ we have $\wt(f) \leq \wt(\sigma(f))$.
	\end{enumerate}
	Then:
	\begin{enumerate}
		\item $\wt(A) \leq \wt((A\setminus O) \cup P)$.
		\item If we can find an $f\in R$ such that $\wt(f) < \wt(\sigma(f))$, then $\wt(A) < \wt((A\setminus O) \cup P)$ and hence $A$ is not optimal.
	\end{enumerate}
\end{lem}
\begin{proof}
	First, reflections preserve weight and symmetrization preserves size, which gives that $|O|=|R|$ and $\wt(O) = \wt(R)$.
	Next, $|O|=|P|$ because $\sigma$ is a bijection, whence $A$ and $(A\setminus O) \cup P$ are downsets of the same size.
	However, we also know that $\wt(R) \leq \wt(P)$ by definition of $\sigma$.
	So, we have $\wt(O) \leq \wt(P)$.
	Thus, $wt((A\setminus O) \cup P) = \wt(A) -\wt(O) + \wt(P) \geq \wt(A) +0 = \wt(A)$.
	Therefore, $\wt(A) \leq \wt((A\setminus O) \cup P)$.
	This proves the first claim.
	For the second claim, if such an $f\in R$ exists then all the inequalities above become strict and hence $A$ is not optimal. 
\end{proof}

The power of the reflect-push method can be observed in Theorem \ref{rectangleProof}.
It is the main tool for its proof.

\begin{thm}\label{rectangleProof}
	Suppose that $\mathscr{M} = \mathscr{M}_{[2]}(\ell_1,\ell_2) = \mathscr{T}_1 \times \mathscr{T}_2$ 
	is a finite multiset lattice with a rank increasing and rank constant weight function.
	If $A\subseteq \mathscr{M}$ is an optimal downset then one of the following statements must hold:
	\begin{enumerate}[label=\textbf{T.\arabic*}]
		\item\label{type1} $A$ is an initial segment of a domination order: that is, there exists $\pi\in \mathfrak{S}_2$ such that $A=\mathscr{M}_{\pi}^{-1}[|A|]$.
		\item\label{type2} $A$ is a symmetrization (using a packed poset) of an initial segment of a domination order: that is,
		there exist $\pi\in \mathfrak{S}_2$, a packed poset $\mathscr{Q}\subseteq \mathscr{M}$ and coordinates $c_1,c_2\in [2]$,
		such that $A = \Sym_{\mathscr{M}}(\mathscr{Q},\mathscr{M}_{\pi}^{-1}[|A|], c_1,c_2 )$.
	\end{enumerate}
\end{thm}
\begin{proof}
	If $A= \mathscr{M}$ or $A= \emptyset$, and $\ell_1 = 1$ or $\ell_2=1$, then $A = \mathscr{M}^{-1}_{\mathcal{L}}[|A|]=\mathscr{M}^{-1}_{\mathcal{C}}[|A|]$.
	So, we suppose $A\neq \mathscr{M}$ and $A\neq \emptyset$, and we have $\ell_1 > 1$ and $\ell_2>1$.
	Let $x\in \mathscr{T}_1$ be the smallest element such that $\mathscr{M}_{\{2\}}(x)\cap A \neq \mathscr{M}_{\{2\}}(x)$.
	Similarly, let $y\in \mathscr{T}_2$ be the smallest element such that $\mathscr{M}_{\{1\}}(y)\cap A \neq \mathscr{M}_{\{1\}}(y)$.
	Then $(x,y)$ is the unique element in $\mathscr{M}_{\{2\}}(x)\cap \mathscr{M}_{\{1\}}(y)$.
	This general setup can be observed in Figure \ref{generalSetupRectangleProof}.
	Let $V$ be all the elements in $\mathscr{M}_{\{2\}}(x)$ after $(x,y)$,
	and $H$ be all the elements in $\mathscr{M}_{\{1\}}(y)$ after $(x,y)$ (Figure \ref{generalSetupRectangleProof_V_and_H}).
	With this setup, we will continue the proof of Theorem \ref{rectangleProof} after we have proved lemmas \ref{case1.1}--\ref{case2.4}.
	
	\begin{figure}
		\centering
		\includegraphics[width=0.5\textwidth]{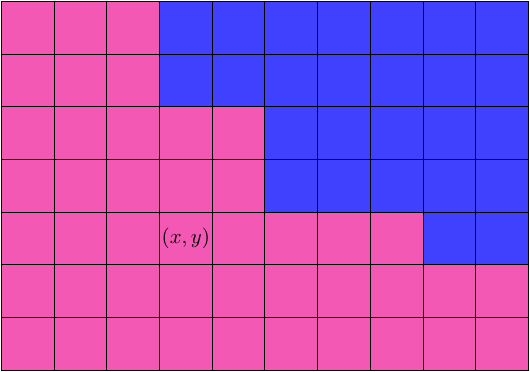}
		\caption{The general setup for the proof of Theorem \ref{rectangleProof}.}
		\label{generalSetupRectangleProof}
	\end{figure}
	
	\begin{figure}
		\centering
		\begin{subfigure}[t]{0.32\textwidth}
			\includegraphics[width=\textwidth]{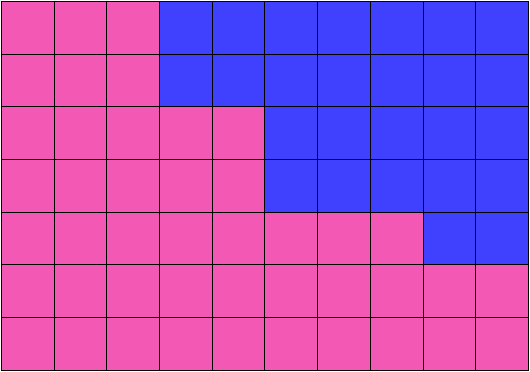}
			\caption{The set $A$.}
		\end{subfigure}
		\hfill
		\begin{subfigure}[t]{0.32\textwidth}
			\includegraphics[width=\textwidth]{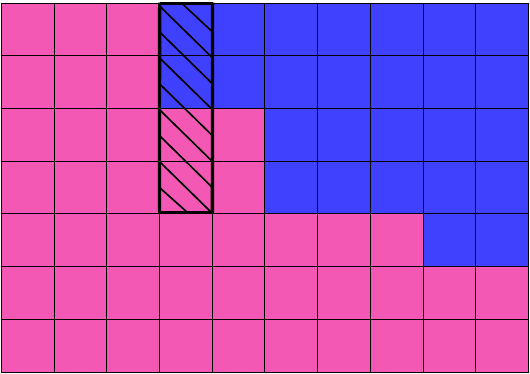}
			\caption{The set $V$ (shaded).}
		\end{subfigure}
		\hfill
		\begin{subfigure}[t]{0.32\textwidth}
			\includegraphics[width=\textwidth]{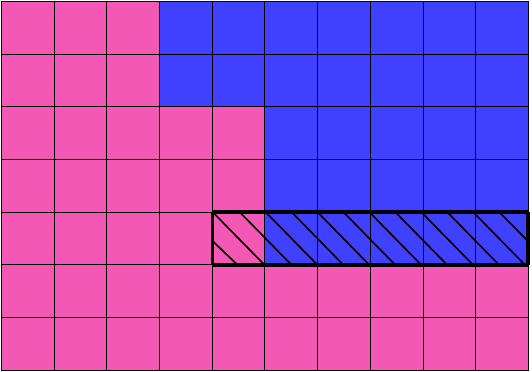}
			\caption{The set $H$ (shaded).}
		\end{subfigure}
		\caption{The sets $A$, $V$ and $H$.}
		\label{generalSetupRectangleProof_V_and_H}
	\end{figure}
	
	\begin{lem}\label{case1.1}
		Suppose that $H\cap A=\emptyset$ or $V\cap A=\emptyset$.
		If $(x,y)\in A$ then $x=0$ or $y=0$ and hence $A$ satisfies the conclusion of Theorem \ref{rectangleProof}.
	\end{lem}
	\begin{proof}
		Without loss of generality suppose that $V\cap A =\emptyset$.
		First, suppose that $|H| \geq |V|$.
		If $x=0$ then $A=\mathscr{M}^{-1}_{\mathcal{C}}[|A|]$ (Figure \ref{Case_1_1_x=0}).
		\begin{figure}
			\centering
			\includegraphics[width=0.5\textwidth]{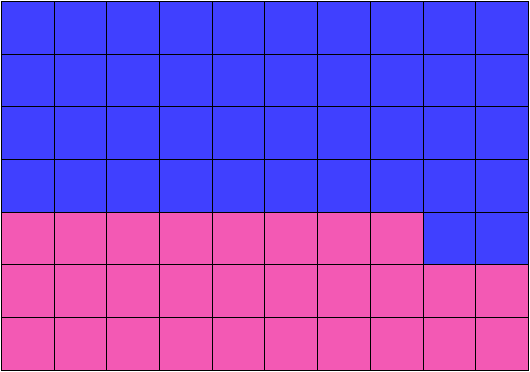}
			\caption{Lemma \ref{case1.1} with $V\cap A =\emptyset$, $|H| \geq |V|$ and $x=0$.}
			\label{Case_1_1_x=0}
		\end{figure}
		
		Assume for the purposes of a contradiction that $x>0$.
		Let
		\begin{align*}
			n = \min\{|H\setminus A|, |V|\}.
		\end{align*}
		We are going to construct a a new set $A'$ from $A$ by using the reflect-push method.
		The first step is a reflection and the second step is a push.
		
		\textit{Lemma \ref{case1.1} Reflect:} Define the packed poset (Figure \ref{Case_1_1_setup_rectangleProof})
		\begin{align*}
			\mathscr{Q} = \{(a,b) \bigm | x-1\leq a \leq x-1+|V| \text{ and } y\leq b \leq y+|V|. \}
		\end{align*}
		Let $O$ be the last $n$ elements in $\mathscr{M}_{\{2\}}(x-1)$ (Figure \ref{Case_1_1_setup_rectangleProof}).
		Then we define the set of reflections of elements in $O$ inside $\mathscr{Q}$ to be (Figure \ref{Case_1_1_setup_rectangleProof})
		\begin{align*}
			R = \{f_{\{1,2\}}\in \mathscr{Q} \bigm | f\in O\cap \mathscr{Q}\}
		\end{align*}
		
		\textit{Lemma \ref{case1.1} Push:} We are no going to push the elements in $R$ forward in $\mathscr{T}_2$.
		Define the set $P$ (Figure \ref{Case_1_1_setup_rectangleProof}) to be the first $n$ elements in $H\setminus A$.
		Then $|P|=|R|$ and the minimum element of $P$ comes after the minimum element in $R$, since $(x,y)\in A$.
		Therefore, the weight of $P$ is strictly greater than the weight of $R$.
		
		\begin{figure}
			\centering
			\begin{subfigure}[t]{0.32\textwidth}
				\includegraphics[width=\textwidth]{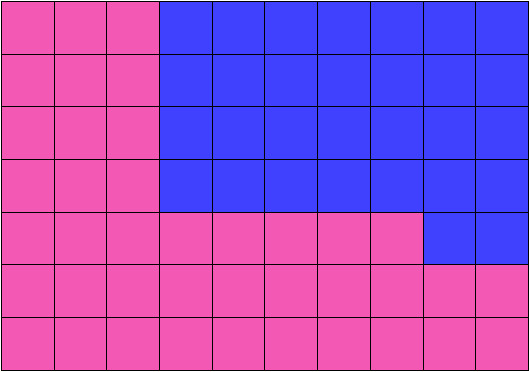}
				\caption{The set $A$.}
			\end{subfigure}
			\hfill
			\begin{subfigure}[t]{0.32\textwidth}
				\includegraphics[width=\textwidth]{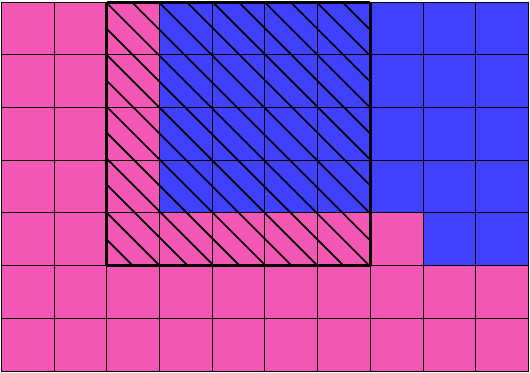}
				\caption{The packed poset $\mathscr{Q}$ (shaded).}
			\end{subfigure}
			\hfill
			\begin{subfigure}[t]{0.32\textwidth}
				\includegraphics[width=\textwidth]{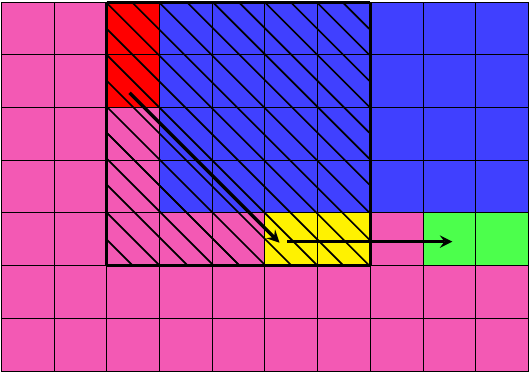}
				\caption{The sets $O$ (start of first arrow/red), $R$ (start of second arrow/yellow) and $P$ (end of second arrow/green).}
			\end{subfigure}
			\caption{Lemma \ref{case1.1} with $V\cap A =\emptyset$, $|H| \geq |V|$, $x>0$ and $n = |H\setminus A|$.}
			\label{Case_1_1_setup_rectangleProof}
		\end{figure}
		
		Finally, we form the set $A'$ (Figure \ref{Case_1_1_rectangleProof}) as follows
		\begin{align*}
			A' = (A\setminus O) \cup P.
		\end{align*}
		By Lemma \ref{reflect-push} the weight of $A'$ is strictly grater than the weight of $A$, which gives a contradiction.
		\begin{figure}
			\centering
			\begin{subfigure}[t]{0.4\textwidth}
				\includegraphics[width=\textwidth]{section2/2D_Pictures/cube_diagram_7x10_proof_case_1_1_H_geq_V.pdf}
				\caption{The set $A$.}
			\end{subfigure}
			\begin{subfigure}[b]{0.1\textwidth}
				\centering
				\raisebox{1.4\hsize}{\includegraphics[width=\textwidth]{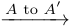}}
			\end{subfigure}
			\begin{subfigure}[t]{0.4\textwidth}
				\includegraphics[width=\textwidth]{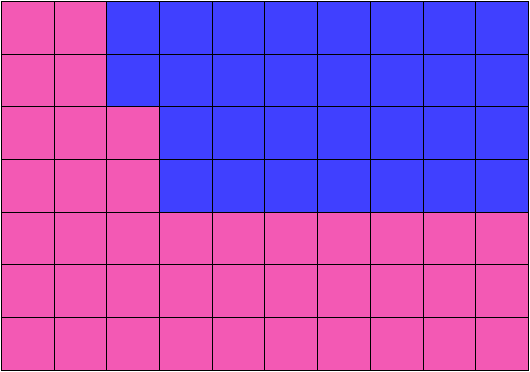}
				\caption{The set $A'$.}
			\end{subfigure}
			\caption{Lemma \ref{case1.1} with $V\cap A =\emptyset$, $|H| \geq |V|$, $x>0$ and $n = |H\setminus A|$: constructing $A'$, see Figure \ref{Case_1_1_setup_rectangleProof} for the individual steps.}
			\label{Case_1_1_rectangleProof}
		\end{figure}
		
		Next, suppose that $|H|< |V|$.
		If $y=0$ then we have that $A$ is a symmetrization of $\mathscr{M}^{-1}_{\mathcal{L}}[|A|]$ (Figure \ref{Case_1_1_y=0}).
		\begin{figure}
			\centering
			\includegraphics[width=0.5\textwidth]{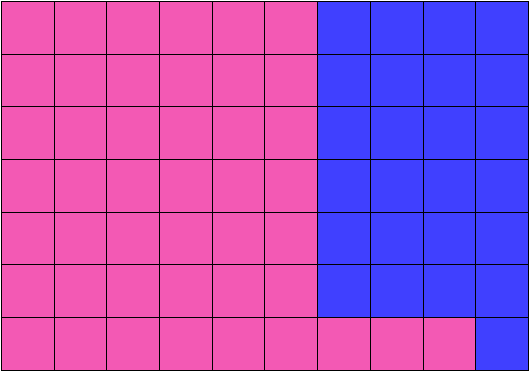}
			\caption{Lemma \ref{case1.1} with $V\cap A =\emptyset$, $|H| < |V|$ and $y=0$.}\label{Case_1_1_y=0}
		\end{figure}
		Assume to the contrary that $y>0$. 
		Then we form a new set $B$ (Figure \ref{Case_1_1_2_rectangleProof}) from $A$ by replacing $H\cap A$ with the first $|H\cap A|$ elements of $V$,
		and we now use a similar reflect-push method like in the case when $|H|\geq |V|$, to show that $B$ is not optimal (giving us a contradiction),
		with $V$ and $H$ exchanging places for $B$ compared to the $x>0$ case.
		\begin{figure}
			\centering
			\begin{subfigure}[t]{0.4\textwidth}
				\includegraphics[width=\textwidth]{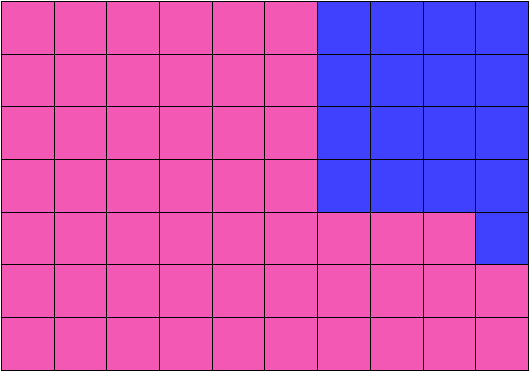}
				\caption{The set $A$.}
			\end{subfigure}
			\begin{subfigure}[b]{0.1\textwidth}
				\centering
				\raisebox{1.4\hsize}{\includegraphics[width=\textwidth]{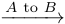}}
			\end{subfigure}
			\begin{subfigure}[t]{0.4\textwidth}
				\includegraphics[width=\textwidth]{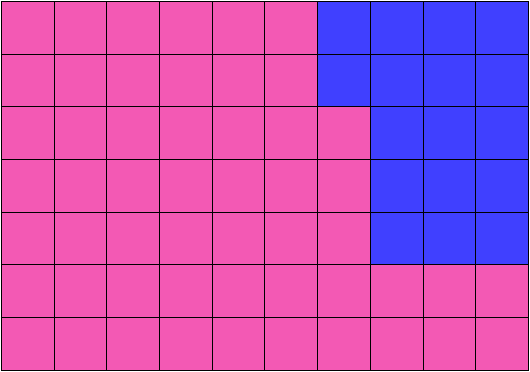}
				\caption{The set $B$.}
			\end{subfigure}
			\caption{Lemma \ref{case1.1} with $V\cap A =\emptyset$, $|H| < |V|$ and $y>0$: constructing $B$.}
			\label{Case_1_1_2_rectangleProof}
		\end{figure}
	\end{proof}
	
	\begin{lem}\label{case1.2}
		Suppose that $H\cap A=\emptyset$ or $V\cap A=\emptyset$.
		If $(x,y)\not\in A$ then $x\leq 1$ or $y\leq 1$ and hence $A$ satisfies the conclusion of Theorem \ref{rectangleProof}.
	\end{lem}
	\begin{proof}
		Without loss of generality suppose that $V\cap A =\emptyset$ and $|H|\geq |V|$.
		If $x\leq 1$ then $A$ is a symmetrization of $\mathscr{M}^{-1}_{\mathcal{C}}[|A|]$ (Figure \ref{Case_1_2_rectangleProofOptimal}).
		\begin{figure}
			\centering
			\begin{subfigure}[t]{0.45\textwidth}
				\includegraphics[width=\textwidth]{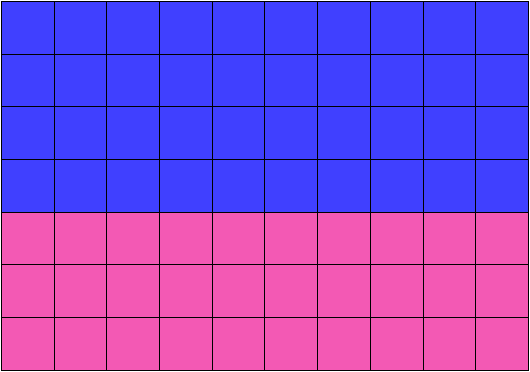}
				\caption{The case $x=0$.}
			\end{subfigure}
			\hfill
			\begin{subfigure}[t]{0.45\textwidth}
				\includegraphics[width=\textwidth]{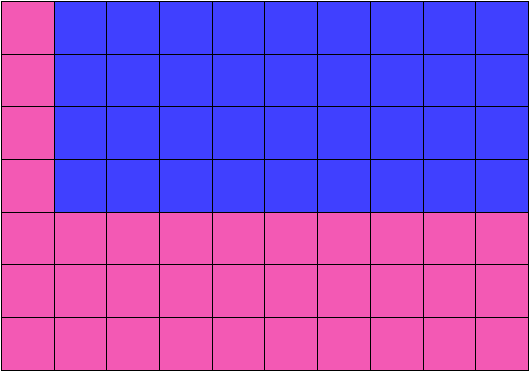}
				\caption{The case $x=1$.}
			\end{subfigure}
			\caption{Lemma \ref{case1.2} with $V\cap A =\emptyset$ and $|H|\geq |V|$.}
			\label{Case_1_2_rectangleProofOptimal}
		\end{figure}
		Assume for the purposes of a contradiction that $x>1$.
		Then we form a new set $B$ (Figure \ref{Case_1_2_rectangleProof}) by removing from $A$ the last $|V|$ elements in $\mathscr{M}_{2}(x-1)$,
		and replacing them with the first elements $|V|$ elements in $\{(x,y)\}\cup H$.
		Note that $B$ is a downset that has the same weight as $A$, and falls under case in Lemma \ref{case1.2},
		which gives that $B$ is not optimal and thus $A$ is not optimal, a contradiction.
		\begin{figure}
			\centering
			\begin{subfigure}[t]{0.4\textwidth}
				\includegraphics[width=\textwidth]{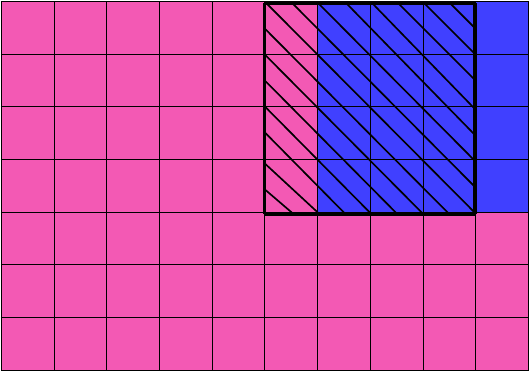}
				\caption{The set $A$.}
			\end{subfigure}
			\begin{subfigure}[b]{0.1\textwidth}
				\centering
				\raisebox{1.4\hsize}{\includegraphics[width=\textwidth]{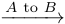}}
			\end{subfigure}
			\begin{subfigure}[t]{0.4\textwidth}
				\includegraphics[width=\textwidth]{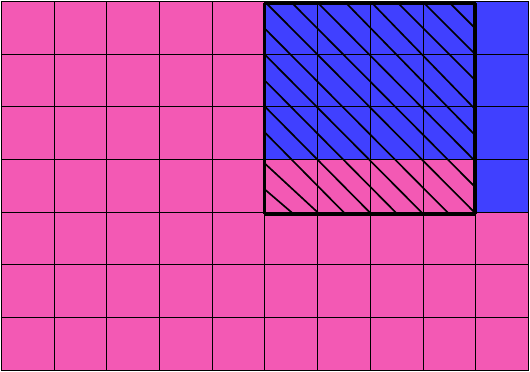}
				\caption{The set $B$}
			\end{subfigure}
			\caption{Lemma \ref{case1.2} with $V\cap A =\emptyset$ and $|H|\geq |V|$: constructing $B$.}
			\label{Case_1_2_rectangleProof}
		\end{figure}
	\end{proof}
	
	Putting Lemma \ref{case1.1} and \ref{case1.2} we see that Theorem \ref{rectangleProof} holds when $H\cap A=\emptyset$ or $V\cap A=\emptyset$.
	So, we need to handle the case when $V \cap A \neq \emptyset$ and $H \cap A \neq \emptyset$.
	This will again be accomplished by a series of reflect-push attacks.
	Let $p = \min\{|V\cap A|, |H\cap A|\}$.
	
	\begin{lem}\label{case2.1}
		Suppose that $V \cap A \neq \emptyset$ and $H \cap A \neq \emptyset$.
		Then $(x+p, y+p) \in A$.
	\end{lem}
	\begin{proof}
		Assume to the contrary that this is not the case.
		We are going to construct a new downset that has a bigger weight than $A$ in two steps.
		This will be another reflect-push method, but the pushing operation will be more advanced now.
		Without loss of generality we can assume that $|V\cap A|\geq |H\cap A|$.
		
		\textit{Lemma \ref{case2.1} Reflect:}
		Define the packed poset (Figure \ref{Case_2_1_setup_rectangleProof})
		\begin{align*}
			\mathscr{Q} = \{(a,b)\in \mathscr{M}\bigm | x\leq a \leq x+p, y \leq b \leq y+p\}.
		\end{align*}
		\begin{figure}
			\centering
			\begin{subfigure}[t]{0.32\textwidth}
				\includegraphics[width=\textwidth]{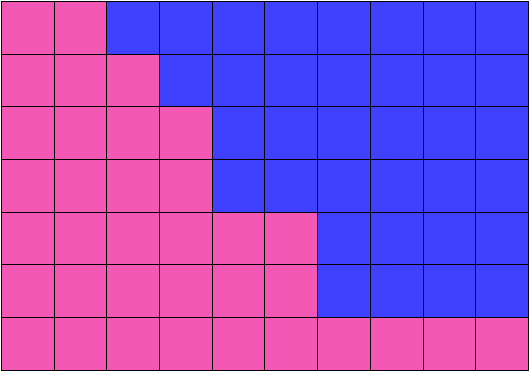}
				\caption{The set $A$.}
			\end{subfigure}
			\hfill
			\begin{subfigure}[t]{0.32\textwidth}
				\includegraphics[width=\textwidth]{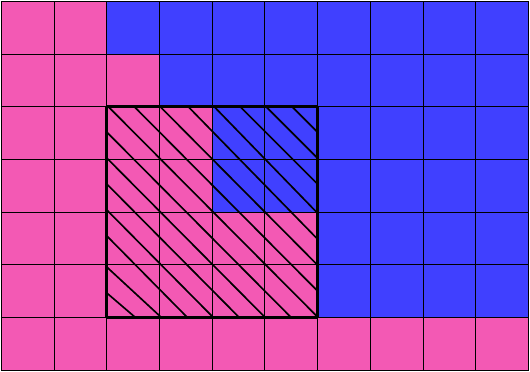}
				\caption{the packed poset $\mathscr{Q}$ (shaded).}
			\end{subfigure}
			\hfill
			\begin{subfigure}[t]{0.32\textwidth}
				\includegraphics[width=\textwidth]{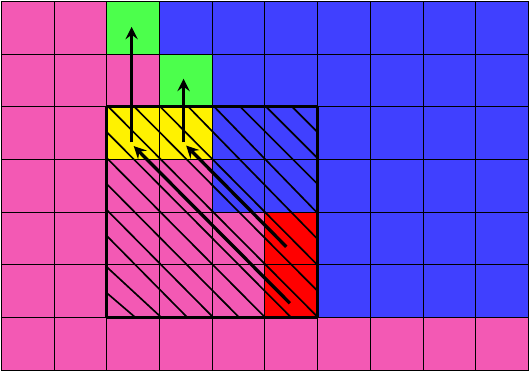}
				\caption{The sets $O$ (start of first arrows/red), $R$ (start of second arrows/yellow) and $P$ (end of second arrows/green).}
			\end{subfigure}
			
			\caption{Lemma \ref{case2.1} with $|V\cap A|\geq |H\cap A|$.}
			\label{Case_2_1_setup_rectangleProof}
		\end{figure}
		Without loss of generality we can assume that $A = \Sym_{\mathscr{M}}(\mathscr{Q},A,2,1)$,
		if not then we replace $A$ with $\Sym_{\mathscr{M}}(\mathscr{Q},A,2,1)$ which has the same weight, 
		and we show that $\Sym_{\mathscr{M}}(\mathscr{Q},A,2,1)$ is not optimal.
		Put (Figure \ref{Case_2_1_setup_rectangleProof})
		\begin{align*}
			O = A\cap \mathscr{M}_{\{2\}}(x+p).
		\end{align*}
		Next, we define $R$ (Figure \ref{Case_2_1_setup_rectangleProof}) to be the set of all reflections (in $\mathscr{Q}$) of elements in $O$.
		
		\textit{Lemma \ref{case2.1} Push:} For each $f\in R$ define the pushed-reflection by $f_{\{1,2\}}'$,
		where $f_{\{1,2\}}'$ is the first element not in $A$ that is above (larger second coordinate) $f_{\{1,2\}}$ (this is the reflection of $f$ in $\mathscr{Q}$).
		Note that $f_{\{1,2\}} \in A$ for every $f\in O$, since we assumed $A = \Sym_{\mathscr{M}}(\mathscr{Q},A,2,1)$.
		Let $P$ (Figure \ref{Case_2_1_setup_rectangleProof}) be the set of all pushed-reflections of elements in $R$.
		Well, $|P|=|R|$.
		Also, every pushed reflection in $P$ has weight that is larger than the original element in $R$.
		
		We then define the set (Figure \ref{Case_2_1_rectangleProof})
		\begin{align*}
			A'=(A\setminus O)\cup P.
		\end{align*}
		\begin{figure}
			\centering
			\begin{subfigure}[t]{0.4\textwidth}
				\includegraphics[width=\textwidth]{section2/2D_Pictures/cube_diagram_7x10_proof_Case_2_1.pdf}
				\caption{The set $A$.}
			\end{subfigure}
			\begin{subfigure}[b]{0.1\textwidth}
				\centering
				\raisebox{1.4\hsize}{\includegraphics[width=\textwidth]{section2/2D_Pictures/Case_1_1_rectangleProof_arrow.pdf}}
			\end{subfigure}
			\begin{subfigure}[t]{0.4\textwidth}
				\includegraphics[width=\textwidth]{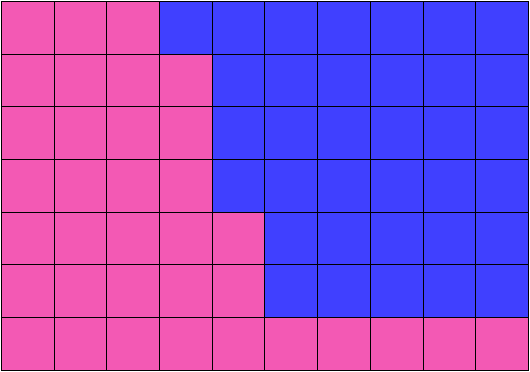}
				\caption{The set $A'$.}
			\end{subfigure}
			
			\caption{Lemma \ref{case2.1} with $|V\cap A|\geq |H\cap A|$: constructing $A'$, see Figure \ref{Case_2_1_setup_rectangleProof} for individual steps.}
			\label{Case_2_1_rectangleProof}
		\end{figure}
		By Lemma \ref{reflect-push} we have that $A'$ has larger weight than $A$.
		Thus, $A$ is not optimal, a contradiction.
		Therefore, $(x+p, y+p) \in A$.
	\end{proof}
	
	Next, we define $(x+p,v_y)\in \mathscr{M}$ to be first element in $\mathscr{M}_{\{2\}}(x+p)\setminus A$.
	Similarly, we define $(v_x,y+p)\in \mathscr{M}$ to be first element in $\mathscr{M}_{\{1\}}(y+p)\setminus A$.
	From Lemma \ref{case2.1} we know that $v_y> y+p$ and $v_x>x+p$.
	
	\begin{lem}\label{case2.2}
		Suppose that $V \cap A \neq \emptyset$ and $H \cap A \neq \emptyset$.
		We have
		\begin{enumerate}
			\item If $|V\cap A|\geq |H\cap A|$ then $(x, v_y) \not\in A$.
			\item If $|H\cap A|\geq |V\cap A|$ then $(v_x, y) \not\in A$.
		\end{enumerate}
	\end{lem}
	\begin{proof}
		We handle the first claim with $|V\cap A| \geq |H\cap A|$ as the second claim follows by symmetry.
		Assume to the contrary that $(x, v_y) \in A$.
		Note that we have $v_y < \ell_1-1$ because of the definition of $x$.
		We are going to construct a new downset that has a bigger weight than $A$ in two steps.
		This will be another reflect-push method.
		The reflection is going to use a different, but similar packed poset to the one in Lemma \ref{case2.1}.
		The pushing operation will be similar to the one in Lemma \ref{case2.1}, but even more advanced.
		
		\textit{Lemma \ref{case2.2} Reflect:}
		Define the packed poset (Figure \ref{Case_2_2_setup_rectangleProof})
		\begin{align*}
			\mathscr{Q} = \{(a,b)\in \mathscr{M} \bigm | x \leq a \leq x+p \text{ and } v_y-p \leq b \leq v_y\}.
		\end{align*}
		\begin{figure}
			\centering
			\begin{subfigure}[t]{0.32\textwidth}
				\includegraphics[width=\textwidth]{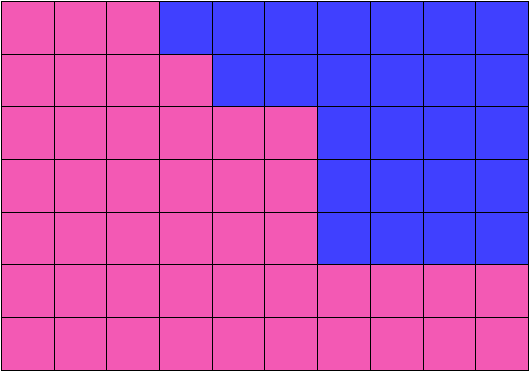}
				\caption{The set $A$.}
			\end{subfigure}
			\hfill
			\begin{subfigure}[t]{0.32\textwidth}
				\includegraphics[width=\textwidth]{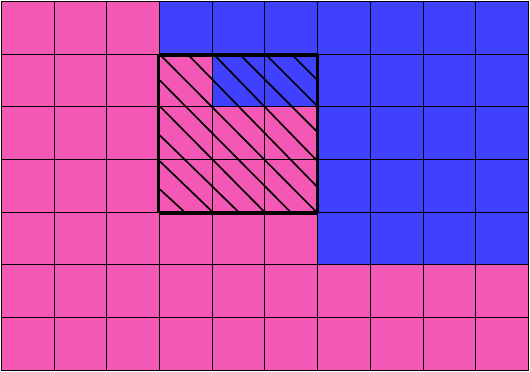}
				\caption{The packed poset $\mathscr{Q}$ (shaded).}
			\end{subfigure}
			\hfill
			\begin{subfigure}[t]{0.32\textwidth}
				\includegraphics[width=\textwidth]{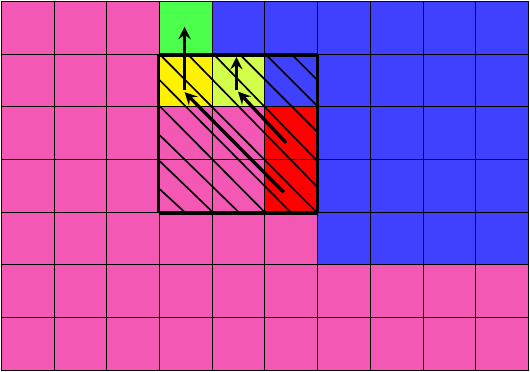}
				\caption{The sets $O$ (start of first arrows/red), $R$ (start of second arrows/yellow) and $P$ (end of second arrows/green).}
			\end{subfigure}
			
			\caption{Lemma \ref{case2.2} with $|V\cap A|\geq |H\cap A|$. Notice that for one box (lime colored) the second arrow starts and ends inside of it, this is to denote that it is not pushed upwards.}
			\label{Case_2_2_setup_rectangleProof}
		\end{figure}
		Put (Figure \ref{Case_2_2_setup_rectangleProof})
		\begin{align*}
			O = A\cap \mathscr{M}_{2}(x+p)\cap \mathscr{Q}.
		\end{align*}
		Next, we define $R$ (Figure \ref{Case_2_2_setup_rectangleProof}) to be the set of all reflections (in $\mathscr{Q}$) of elements in $O$.
		
		\textit{Lemma \ref{case2.2} Push:} For each $f\in R$ define the pushed-reflection by $f_{\{1,2\}}'$,
		where $f_{\{1,2\}}'$ is the first element not in $A$ that is $f_{\{1,2\}}$ or is above (larger second coordinate) $f_{\{1,2\}}$ (this is the reflection of $f$ in $\mathscr{Q}$).
		Let $P$ (Figure \ref{Case_2_2_setup_rectangleProof}) be the set of all pushed-reflections of elements in $R$.
		Well, $|P|=|R|$.
		Also, $(x,v_y)\in A$, whence the pushed reflection of $(x,v_y) \in R$ is above $(x,v_y)$, and thus has larger weight.
		
		We then define the set (Figure \ref{Case_2_2_rectangleProof})
		\begin{align*}
			A'=(A\setminus O)\cup P.
		\end{align*}
		\begin{figure}
			\centering
			\begin{subfigure}[t]{0.4\textwidth}
				\includegraphics[width=\textwidth]{section2/2D_Pictures/cube_diagram_7x10_proof_Case_2_2.pdf}
				\caption{The set $A$.}
			\end{subfigure}
			\begin{subfigure}[b]{0.1\textwidth}
				\centering
				\raisebox{1.4\hsize}{\includegraphics[width=\textwidth]{section2/2D_Pictures/Case_1_1_rectangleProof_arrow.pdf}}
			\end{subfigure}
			\begin{subfigure}[t]{0.4\textwidth}
				\includegraphics[width=\textwidth]{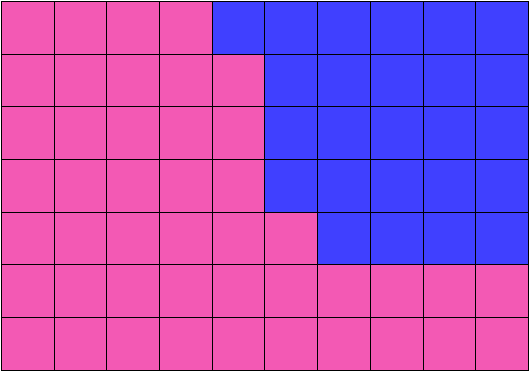}
				\caption{The set $A'$.}
			\end{subfigure}
			
			\caption{Lemma \ref{case2.2} with $|V\cap A|\geq |H\cap A|$: constructing $A'$, see Figure \ref{Case_2_2_setup_rectangleProof} for individual steps.}
			\label{Case_2_2_rectangleProof}
		\end{figure}
		By Lemma \ref{reflect-push} we have that $A'$ has larger weight than $A$.
		Thus, $A$ is not optimal, a contradiction.
		Therefore, $(x, v_y) \not\in A$.
	\end{proof}

	\begin{lem}\label{case2.3}
		Suppose that $V \cap A \neq \emptyset$ and $H \cap A \neq \emptyset$.
		We have:
		\begin{enumerate}
			\item If $|V\cap A|\geq |H\cap A|$ then $v_y = \ell_2-1$.
			\item If $|H\cap A|\geq |V\cap A|$ then $v_x = \ell_1-1$.
		\end{enumerate}
	\end{lem}
	\begin{proof}
		We handle the first claim with $|V\cap A|\geq |H\cap A|$ as the second claim follows by symmetry.
		Assume to the contrary that $v_y < \ell_2-1$.
		By Lemma \ref{case2.2} we have that $(x,v_y)\not\in A$.
		Then we can just do a symmetrization (Figure \ref{Case_2_3_rectangleProof}) and construct a set $B$ with the same weight as $A$, 
		but now Lemma \ref{case2.1} or Lemma \ref{case2.2} gives us that $B$ is not optimal, whence $A$ is not optimal.
			\begin{figure}
				\centering
				\begin{subfigure}[t]{0.4\textwidth}
					\includegraphics[width=\textwidth]{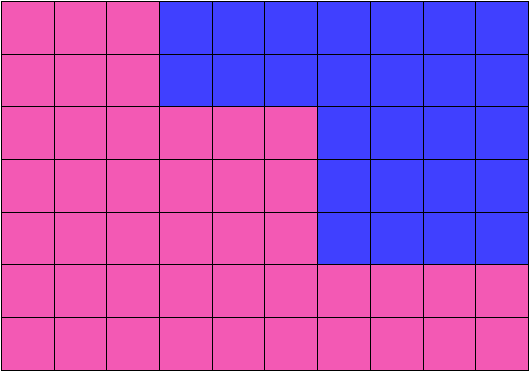}
					\caption{The set $A$.}
				\end{subfigure}
				\begin{subfigure}[b]{0.1\textwidth}
					\centering
					\raisebox{1.4\hsize}{\includegraphics[width=\textwidth]{section2/2D_Pictures/Case_1_2_rectangleProof_arrow.pdf}}
				\end{subfigure}
				\begin{subfigure}[t]{0.4\textwidth}
					\includegraphics[width=\textwidth]{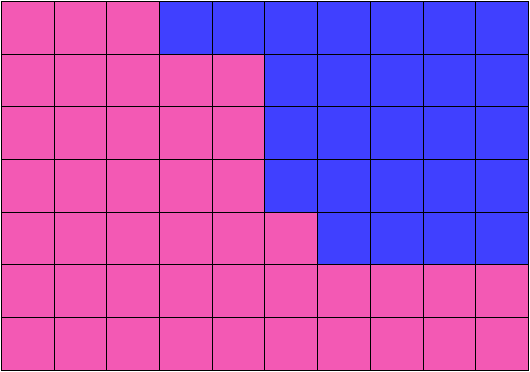}
					\caption{The set $B$.}
				\end{subfigure}
				\caption{Lemma \ref{case2.3} with $|V\cap A|\geq |H\cap A|$ and $v_y = \ell_2-2<\ell_2-1$: constructing $B$, and $B$ is not optimal by Lemma \ref{case2.1}}
				\label{Case_2_3_rectangleProof}
			\end{figure}
	\end{proof}

	\begin{lem}\label{case2.4}
		Suppose that $V \cap A \neq \emptyset$ and $H \cap A \neq \emptyset$.
		We have:
		\begin{enumerate}
			\item If $|V\cap A|\geq |H\cap A|$ then $x= 0$.
			\item If $|H\cap A|\geq |V\cap A|$ then $y= 0$.
		\end{enumerate}
	\end{lem}
	\begin{proof}
		We prove the first claim with $|V\cap A|\geq |H\cap A|$, as the second claim follows by symmetry.
		Assume to the contrary that $x>0$.
		Note that $v_y=\ell_2 - 1$ by Lemma \ref{case2.3}.
		Then we can do a symmetrization (Figure \ref{Case_2_4_rectangleProof}) to get a set $B$ of the same weight as $A$,
		but $B$ will not be optimal by Lemma \ref{case1.1}.
		Note that even if $|V\cap A|=|H\cap A|$, we always need to fill up $|H\cap A|$ elements, but we also always have at least $|V\cap A|+1$ elements to use.
		\begin{figure}
			\centering
			\begin{subfigure}[t]{0.4\textwidth}
				\includegraphics[width=\textwidth]{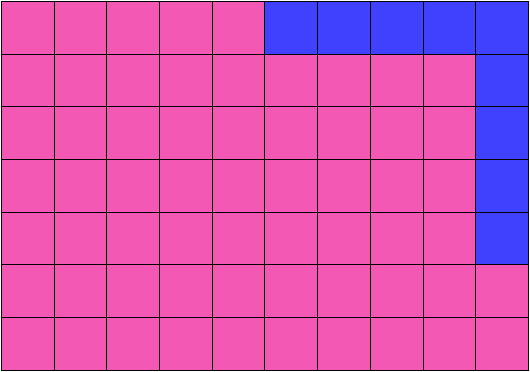}
				\caption{The set $A$.}
			\end{subfigure}
			\begin{subfigure}[b]{0.1\textwidth}
				\centering
				\raisebox{1.4\hsize}{\includegraphics[width=\textwidth]{section2/2D_Pictures/Case_1_2_rectangleProof_arrow.pdf}}
			\end{subfigure}
			\begin{subfigure}[t]{0.4\textwidth}
				\includegraphics[width=\textwidth]{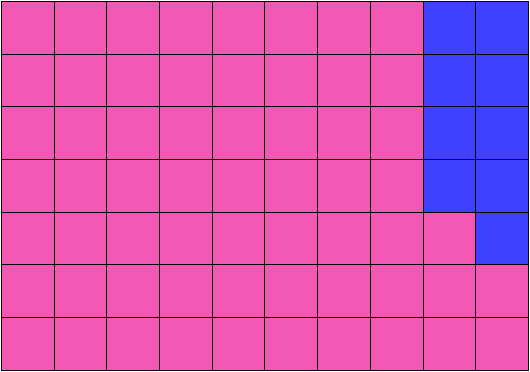}
				\caption{The set $B$.}
			\end{subfigure}
			
			\caption{Lemma \ref{case2.4} with $|V\cap A| \geq |H\cap A|$ and $x> 0$: constructing $B$, and $B$ is not optimal by Lemma \ref{case1.1}.}
			\label{Case_2_4_rectangleProof}
		\end{figure}
	\end{proof}
	
	Now we can finish the proof of Theorem \ref{rectangleProof}.
	Putting together Lemma \ref{case2.1}, \ref{case2.2}, \ref{case2.3} and \ref{case2.4}, 
	we have that $A$ is a symmetrization of $\mathscr{M}^{-1}_{\mathcal{L}}[|A|]$ (Figure \ref{Case_2_5_rectangleProof}) 
	or $A$ is a symmetrization of $\mathscr{M}^{-1}_{\mathcal{C}}[|A|]$, 
	whenever $V \cap A \neq \emptyset$ and $H \cap A \neq \emptyset$.
	Therefore, either \ref{type1} or \ref{type2} from Theorem \ref{rectangleProof} holds, 
	since we handled the case $H\cap A=\emptyset$ or $V\cap A=\emptyset$ with Lemma \ref{case1.1} and Lemma \ref{case1.2}.
	\begin{figure}
		\centering
		\includegraphics[width=0.5\textwidth]{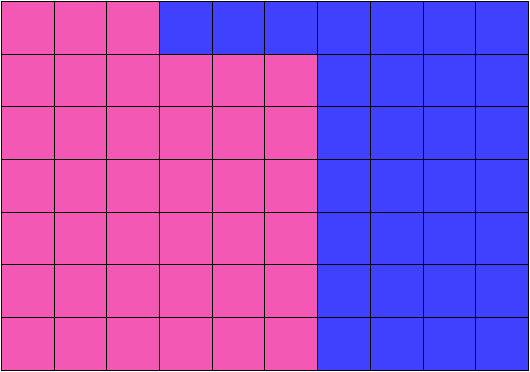}
		\caption{The case $V \cap A \neq \emptyset$ and $H \cap A \neq \emptyset$ when $A$ is a symmetrization of $\mathscr{M}^{-1}_{\mathcal{L}}[|A|]$.}\label{Case_2_5_rectangleProof}
	\end{figure}

\end{proof}

If one follows the proof of Theorem \ref{rectangleProof} carefully,
then a much more specific statement comes out.
The second statement of the theorem, that an optimal set is a symmetrization of a domination order,
can be even more detailed.
We can state exactly which packed posets are used for the symmetrization.
This is captured in Corollary \ref{fullRectangleProof}.

\begin{cor}\label{fullRectangleProof}
	Suppose that $\mathscr{M} = \mathscr{M}_{[2]}(\ell_1,\ell_2)$ is a finite multiset lattice with a rank increasing and rank constant weight function.
	If $A\subseteq \mathscr{M}$ is an optimal downset then exactly one of the following statements must hold:
	\begin{enumerate}[label=\textbf{T.\arabic*}]
		\item\label{type1Cor} $A$ is an initial segment of a domination order: that is, there exists $\pi\in \mathfrak{S}_2$ such that $A=\mathscr{M}_{\pi}^{-1}[|A|]$.
		\item\label{type2Cor}  $A$ is a symmetrization (using a packed poset) of an initial segment of a domination order: that is,
		there exist $\pi\in \mathfrak{S}_2$, a packed poset $\mathscr{Q}\subseteq \mathscr{M}$ and coordinates $c_1,c_2\in [2]$,
		such that $A = \Sym_{\mathscr{M}}(\mathscr{Q},\mathscr{M}_{\pi}^{-1}[|A|], c_1,c_2 )$.
		Furthermore, we can say exactly witch packed poset $\mathscr{Q}$ is used.
		Exactly one of the following two cases must happen:
		\begin{enumerate}[label=\textbf{T.2.\arabic*}]
			\item\label{type2.1}  Let $f\in \mathscr{M}$ be the last element in the order $\mathcal{D}_{\pi}$ such that $f \in \mathscr{M}_{\pi}^{-1}[|A|]$.
			Then put $p = |\mathscr{M}_{\{\pi(2)\}}(f(\pi(1)))\cap  \mathscr{M}_{\pi}^{-1}[|A|]|$ and $q=|\mathscr{M}_{\{\pi(1)\}}(0)\setminus \mathscr{M}_{\pi}^{-1}[|A|]|$.
			If $0<p -1 \leq q$
			then
			\begin{align*}
				\mathscr{Q} = \{g\in \mathscr{M} \bigm | &0\leq g(\pi(2)) \leq f(\pi(2)) \text{ and }\\ 
														&f(\pi(1)) \leq g(\pi(1)) \leq f(\pi(1)) +p-1\}.
			\end{align*}
			\item\label{type2.2}  Let $f\in \mathscr{M}$ be the first element in the order $\mathcal{D}_{\pi}$ such that $f \not\in \mathscr{M}_{\pi}^{-1}[|A|]$.
			Then put $p= |\mathscr{M}_{\{\pi(2)\}}(f(\pi(1)))\setminus \mathscr{M}_{\pi}^{-1}[|A|]|$ and $q = |\mathscr{M}_{\{\pi(1)\}}(\ell_{\pi(2)}-1)\cap \mathscr{M}_{\pi}^{-1}[|A|]|$.
			If $0<p-1 \leq q$ then
			\begin{align*}
				\mathscr{Q} = \{g\in \mathscr{M} \bigm | &f(\pi(2)) + 1\leq g(\pi(2)) \leq \ell_{\{\pi(2)\}}-1 \text{ and }\\ 
														&f(\pi(1))-p-1 \leq g(\pi(1)) \leq f(\pi(1))\}.
			\end{align*}
		\end{enumerate}
	\end{enumerate}
\end{cor}

Theorem \ref{rectangleProof} and Corollary \ref{fullRectangleProof} give the structure of all possible optimal downsets in any $\mathscr{M}(\ell_1,\ell_2)$.
In particular, they strengthen Lindsay's edge-isoperimetric inequality (see Section \ref{appliations}) in two dimensions by providing all cases of optimality 
for downsets and any rank increasing and rank constant weight function,
compared to the optimality of the initial segments of lexicographic order with the standard weight function in the original theorem.

For every set size $m$ the potential optimal sets based on Theorem \ref{rectangleProof} can be divided into equivalence classes.
The equivalence relation for this is that for every set size $m$, two downsets are equivalent if they are both symmetrizations of the same domination order.
In most cases there will be two equivalence classes, one for $\mathcal{L}$ and one for $\mathcal{C}$.
However, note that it could be the case that there is one equivalence class in some cases.
For this example, take $\ell_1=\ell_2$, and notice that the initial segments of the domination orders are symmetrizations of each other.
So, to figure out if a downset of size $m$ is optimal we first need to figure out if it is a symmetrization of a domination order,
if it is not then it is not optimal,
and if it is a symmetrization then we figure out which initial segment of size $m$ of the domination orders is optimal.
We call symmetrizations of initial segments of $\mathcal{L}$, \textit{lexicographic type sets},
and symmetrizations of initial segments of $\mathcal{C}$, \textit{colexicographic type sets}
The next section completely determines when one domination order is better than another, and when both give an optimal downset.

%auto-ignore
\section{Cases of Equality Between Optimal Downsets in Rectangles}\label{rectanglesExact}

\begin{cor}\label{squares}
	Suppose that $\mathscr{M} = \mathscr{M}_{[2]}(\ell_1,\ell_2)$ with a rank increasing and rank constant weight function.
	If $\ell_1=\ell_2$ and $m\in [\ell_1\ell_2+1]_0$ then $\wt(\mathscr{M}_{\mathcal{C}}^{-1} (m)) = \wt(\mathscr{M}_{\mathcal{L}}^{-1} (m))$.
	That is, all lexicographic type sets and all colexicographic type sets are optimal.
\end{cor}
\begin{proof}
	Since $\ell_1=\ell_2$ we have that $\mathscr{M}_{\mathcal{C}}^{-1} (m)$ and $\mathscr{M}_{\mathcal{L}}^{-1} (m)$ are symmetrizations of each other.
	Therefore, the claim follows from Theorem \ref{rectangleProof}.
\end{proof}

\begin{lem}\label{lemmaForAll}
	Suppose that $\mathscr{M} = \mathscr{M}_{[2]}(\ell_1,\ell_2)$ with $1<\ell_1< \ell_2$ and a rank increasing and rank constant weight function.
	Also, let $A\subseteq \mathscr{M}$ be an initial segment of $\mathcal{C}$.
	If $A$ is optimal then $|A| \leq \ell_1$ or $|A|\geq \ell_1(\ell_2-1)$ or $|A| = k\ell_1$ for some $k\in \N$.
\end{lem}
\begin{proof}
	Assume to the contrary that $\ell_1<|A|< \ell_1(\ell_2-1)$ and $|A| \neq k\ell_1$ for all $k\in \N$.
	First, define the packed poset
	\begin{align*}
		\mathscr{Q}_1 = \{(a,b)\in \mathscr{M} \bigm | 0\leq a,b \leq \ell_1-1\}.
	\end{align*}
	If $|A| < \ell_1(\ell_1-1)$ (note that we have $\ell_1 > 2$ in this case) then we consider $\Sym_{\mathscr{M}}(\mathscr{Q}_1, A, 2,1)$.
	We show that $\Sym_{\mathscr{M}}(\mathscr{Q}_1, A, 2,1)$ is not optimal with two separate cases.
	If $(\ell_1-2,\ell_1-2)\not\in \Sym_{\mathscr{M}}(\mathscr{Q}_1, A, 2,1)$ (Figure \ref{Case_1_lemmaForAll}) 
	then $\Sym_{\mathscr{M}}(\mathscr{Q}_1, A, 2,1)$ is not optimal by Lemma \ref{case2.1}.
	If $(\ell_1-2,\ell_1-2)\in \Sym_{\mathscr{M}}(\mathscr{Q}_1, A, 2,1)$ (Figure \ref{Case_1_1_lemmaForAll})
	then $\Sym_{\mathscr{M}}(\mathscr{Q}_1, A, 2,1)$ is not optimal by Lemma \ref{case2.2},
	since $(0,\ell_1-1)\in \Sym_{\mathscr{M}}(\mathscr{Q}_1, A, 2,1)$
	which is not optimal by either Lemma \ref{case2.1} or Lemma \ref{case2.2} since $\ell_1 < \ell_2$,
	whence we have a contradiction.
	\begin{figure}
		\centering
		\begin{subfigure}[t]{0.23\textwidth}
			\includegraphics[width=\textwidth]{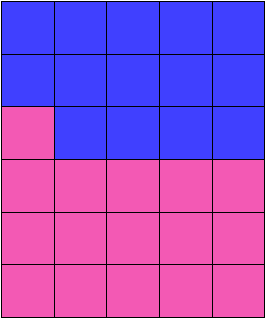}
			\caption{$A$.}
		\end{subfigure}
		\hfill
		\begin{subfigure}[b]{0.1\textwidth}
			\centering
			\raisebox{1.38\hsize}{\includegraphics[width=\textwidth]{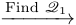}}
		\end{subfigure}
		\hfill
		\begin{subfigure}[t]{0.23\textwidth}
			\includegraphics[width=\textwidth]{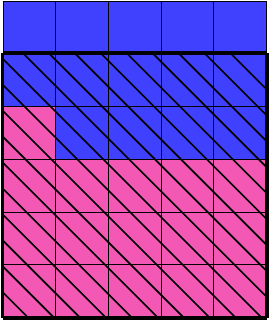}
			\caption{$\mathscr{Q}_1$ (shaded).}
		\end{subfigure}
		\hfill
		\begin{subfigure}[b]{0.18\textwidth}
			\centering
			\raisebox{0.766666667\hsize}{\includegraphics[width=\textwidth]{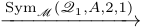}}
		\end{subfigure}
		\hfill
		\begin{subfigure}[t]{0.23\textwidth}
			\includegraphics[width=\textwidth]{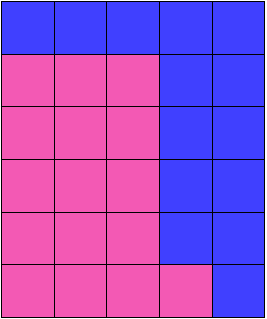}
			\caption{$\Sym_{\mathscr{M}}(\mathscr{Q}_1, A, 2,1)$.}
		\end{subfigure}
		
		\caption{Lemma \ref{lemmaForAll} with $|A| < \ell_1(\ell_1-1)$ and $(\ell_1-2,\ell_1-2)\not\in \Sym_{\mathscr{M}}(\mathscr{Q}_1, A, 2,1)$ : $\Sym_{\mathscr{M}}(\mathscr{Q}_1, A, 2,1)$ is not optimal by Lemma \ref{case2.1}.}
		\label{Case_1_lemmaForAll}
	\end{figure}
	\begin{figure}
		\centering
		\begin{subfigure}[t]{0.23\textwidth}
			\includegraphics[width=\textwidth]{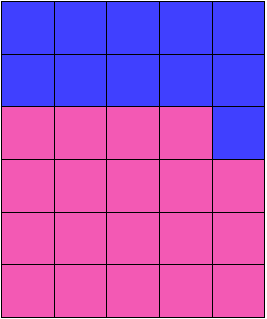}
			\caption{$A$.}
		\end{subfigure}
		\hfill
		\begin{subfigure}[b]{0.1\textwidth}
			\centering
			\raisebox{1.38\hsize}{\includegraphics[width=\textwidth]{section3/2D_Pictures/Case_1_lemmaForAll_arrow_1.pdf}}
		\end{subfigure}
		\hfill
		\begin{subfigure}[t]{0.23\textwidth}
			\includegraphics[width=\textwidth]{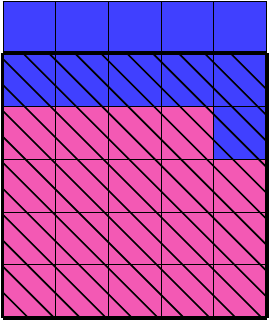}
			\caption{$\mathscr{Q}_1$ (shaded).}
		\end{subfigure}
		\hfill
		\begin{subfigure}[b]{0.18\textwidth}
			\centering
			\raisebox{0.766666667\hsize}{\includegraphics[width=\textwidth]{section3/2D_Pictures/Case_1_lemmaForAll_arrow_2.pdf}}
		\end{subfigure}
		\hfill
		\begin{subfigure}[t]{0.23\textwidth}
			\includegraphics[width=\textwidth]{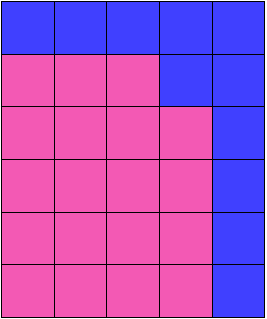}
			\caption{$\Sym_{\mathscr{M}}(\mathscr{Q}_1, A, 2,1)$.}
		\end{subfigure}
		
		\caption{Lemma \ref{lemmaForAll} with $|A| < \ell_1(\ell_1-1)$ and $(\ell_1-2,\ell_1-2)\in \Sym_{\mathscr{M}}(\mathscr{Q}_1, A, 2,1)$ : $\Sym_{\mathscr{M}}(\mathscr{Q}_1, A, 2,1)$ is not optimal by Lemma \ref{case2.2}.}
		\label{Case_1_1_lemmaForAll}
	\end{figure}
	
	So, suppose that $|A| > \ell_1(\ell_1-1)$.
	Let $(\ell_1-1, y)\in \mathscr{M}$ be the first element in $\mathscr{M}_{\{2\}}(\ell_1-1)\setminus A$.
	Define the packed poset
	\begin{align*}
		\mathscr{Q}_2 = \{(a,b)\in \mathscr{M} \bigm | 0\leq a \leq \ell_1-1 \text{ and } y-\ell_1 + 1 \leq b\leq y\}
	\end{align*}
	We handle two cases, one when $\ell_1(\ell_1-1)< |A| < \ell_1\ell_1$ and another when $|A| > \ell_1\ell_1$.
	The only difference between them is that we use $\mathscr{Q}_1$ if $\ell_1(\ell_1-1)< |A| < \ell_1\ell_1$,
	and we use $\mathscr{Q}_2$ if $|A| > \ell_1\ell_1$.
	If $\ell_1(\ell_1-1)< |A| < \ell_1\ell_1$ then we consider then we consider $\Sym_{\mathscr{M}}(\mathscr{Q}_1, A, 2,1)$ (Figure \ref{Case_2_1_lemmaForAll}),
	and $\Sym_{\mathscr{M}}(\mathscr{Q}_1, A, 2,1)$ is not optimal by a simple reflect-push method (Figure \ref{Case_2_1_lemmaForAll}).
	\begin{figure}
		\centering
		\begin{subfigure}[t]{0.23\textwidth}
			\includegraphics[width=\textwidth]{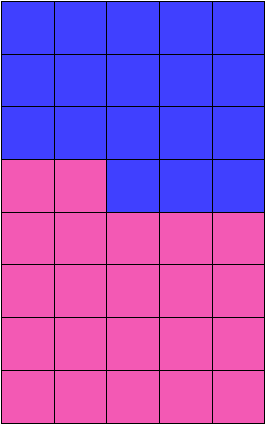}
			\caption{$A$.}
		\end{subfigure}
		\hfill
		\begin{subfigure}[b]{0.1\textwidth}
			\centering
			\raisebox{1.84\hsize}{\includegraphics[width=\textwidth]{section3/2D_Pictures/Case_1_lemmaForAll_arrow_1.pdf}}
		\end{subfigure}
		\hfill
		\begin{subfigure}[t]{0.23\textwidth}
			\includegraphics[width=\textwidth]{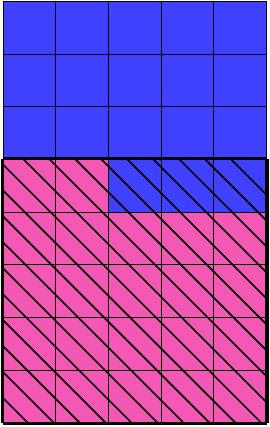}
			\caption{$\mathscr{Q}_1$ (shaded).}
		\end{subfigure}
		\hfill
		\begin{subfigure}[b]{0.18\textwidth}
			\centering
			\raisebox{1.022222222\hsize}{\includegraphics[width=\textwidth]{section3/2D_Pictures/Case_1_lemmaForAll_arrow_2.pdf}}
		\end{subfigure}
		\hfill
		\begin{subfigure}[t]{0.23\textwidth}
			\includegraphics[width=\textwidth]{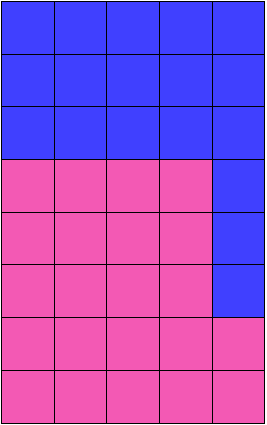}
			\caption{$\Sym_{\mathscr{M}}(\mathscr{Q}_1, A, 2,1)$.}
		\end{subfigure}
		
		\vspace{0.5cm}
		
		\begin{subfigure}[t]{0.23\textwidth}
			\includegraphics[width=\textwidth]{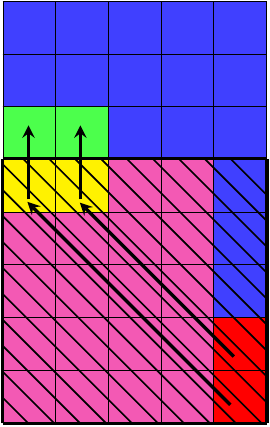}
			\caption{Setup for the reflect-push method.}
		\end{subfigure}
		\begin{subfigure}[b]{0.18\textwidth}
			\centering
			\raisebox{1.022222222\hsize}{\includegraphics[width=\textwidth]{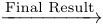}}
		\end{subfigure}
		\begin{subfigure}[t]{0.23\textwidth}
			\includegraphics[width=\textwidth]{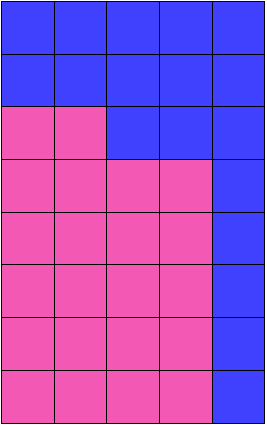}
			\caption{A set which has larger weight than $A$.}
		\end{subfigure}
		
		\caption{Lemma \ref{lemmaForAll} with $\ell_1(\ell_1-1)< |A| < \ell_1\ell_1$: $\Sym_{\mathscr{M}}(\mathscr{Q}_1, A, 2,1)$ is not optimal by a reflect-push method.}
		\label{Case_2_1_lemmaForAll}
	\end{figure}
	If $|A| > \ell_1\ell_1$ then we consider then we consider $\Sym_{\mathscr{M}}(\mathscr{Q}_2, A, 2,1)$ (Figure \ref{Case_2_2_lemmaForAll}),
	and $\Sym_{\mathscr{M}}(\mathscr{Q}_2, A, 2,1)$ is not optimal by a simple reflect-push method (Figure \ref{Case_2_2_lemmaForAll}).
	\begin{figure}
		\centering
		\begin{subfigure}[t]{0.23\textwidth}
			\includegraphics[width=\textwidth]{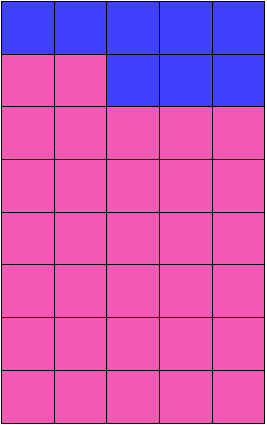}
			\caption{$A$.}
		\end{subfigure}
		\hfill
		\begin{subfigure}[b]{0.1\textwidth}
			\centering
			\raisebox{1.84\hsize}{\includegraphics[width=\textwidth]{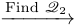}}
		\end{subfigure}
		\hfill
		\begin{subfigure}[t]{0.23\textwidth}
			\includegraphics[width=\textwidth]{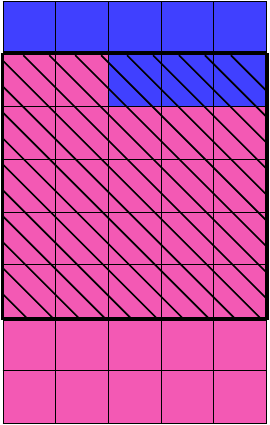}
			\caption{$\mathscr{Q}_2 (shaded)$.}
		\end{subfigure}
		\hfill
		\begin{subfigure}[b]{0.18\textwidth}
			\centering
			\centering
			\raisebox{1.022222222\hsize}{\includegraphics[width=\textwidth]{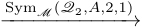}}
		\end{subfigure}
		\hfill
		\begin{subfigure}[t]{0.23\textwidth}
			\includegraphics[width=\textwidth]{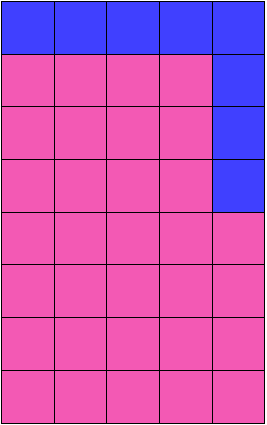}
			\caption{$\Sym_{\mathscr{M}}(\mathscr{Q}_2, A, 2,1)$.}
		\end{subfigure}
		
		\vspace{0.5cm}
		
		\begin{subfigure}[t]{0.23\textwidth}
			\includegraphics[width=\textwidth]{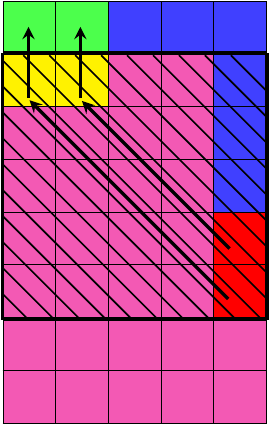}
			\caption{Setup for the reflect-push method.}
		\end{subfigure}
		\begin{subfigure}[b]{0.18\textwidth}
			\centering
			\raisebox{1.022222222\hsize}{\includegraphics[width=\textwidth]{section3/2D_Pictures/Case_2_1_lemmaForAll_arrow_3.pdf}}
		\end{subfigure}
		\begin{subfigure}[t]{0.23\textwidth}
			\includegraphics[width=\textwidth]{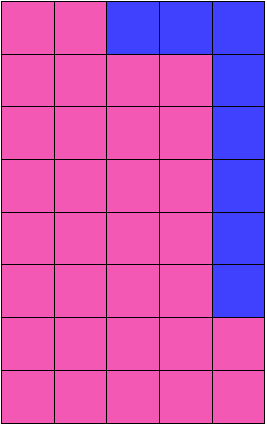}
			\caption{A set with a higher weight than $A$.}
		\end{subfigure}
		
		\caption{Lemma \ref{lemmaForAll} with $|A| > \ell_1\ell_1$: $\Sym_{\mathscr{M}}(\mathscr{Q}_2, A, 2,1)$ is not optimal by a reflect-push method.}
		\label{Case_2_2_lemmaForAll}
	\end{figure}
	Therefore, we get a contradiction in all cases and the claim holds.
\end{proof}

\begin{cor}\label{rectanglesOneMore}
	Suppose that $\mathscr{M} = \mathscr{M}_{[2]}(\ell_1,\ell_2)$ with $1<\ell_1 = \ell_2 -1$ and a rank increasing and rank constant weight function.
	Also, let $A\subseteq \mathscr{M}$ be a downset.
	\begin{enumerate}
		\item Suppose that $A = \mathscr{M}_{\mathcal{C}}^{-1} (|A|)$, 
		then $A$ is optimal if and only if $|A| \leq \ell_1$ or $|A|\geq \ell_1(\ell_2-1)$ or $|A| = k\ell_1$ for some $k\in \N$.
		\item If $A = \mathscr{M}_{\mathcal{L}}^{-1} (|A|)$ then $A$ is optimal.
	\end{enumerate}
\end{cor}
\begin{proof}
	We prove the first claim.
	So, suppose that $A = \mathscr{M}_{\mathcal{C}}^{-1} (|A|)$.
	Define (Figure \ref{Case_1_rectanglesOneMore})
	\begin{align*}
		\mathscr{Q} = \{(a,b)\in \mathscr{M} \bigm | 0\leq a,b \leq \ell_1-1\}.
	\end{align*}
	\begin{figure}
		\centering
		\begin{subfigure}[t]{0.23\textwidth}
			\includegraphics[width=\textwidth]{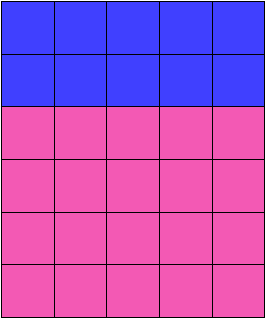}
			\caption{$A$.}
		\end{subfigure}
		\hfill
		\begin{subfigure}[b]{0.1\textwidth}
			\centering
			\raisebox{1.38\hsize}{\includegraphics[width=\textwidth]{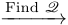}}
		\end{subfigure}
		\hfill
		\begin{subfigure}[t]{0.23\textwidth}
			\includegraphics[width=\textwidth]{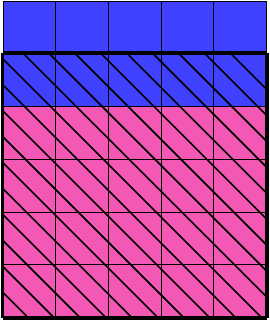}
			\caption{$\mathscr{Q}$ (shaded).}
		\end{subfigure}
		\hfill
		\begin{subfigure}[b]{0.18\textwidth}
			\centering
			\raisebox{0.766666667\hsize}{{\includegraphics[width=\textwidth]{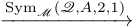}}}
		\end{subfigure}
		\hfill
		\begin{subfigure}[t]{0.23\textwidth}
			\includegraphics[scale=0.8]{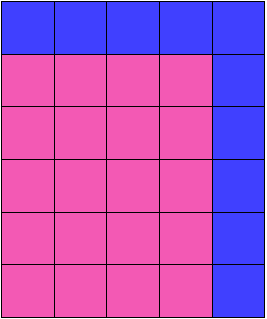}
			\caption{$\Sym_{\mathscr{M}}(\mathscr{Q}, A, 2,1)$.}
		\end{subfigure}
		
		\caption{Corollary \ref{rectanglesOneMore}: $|A|=k\ell_1$ implies $A$ is optimal.}
		\label{Case_1_rectanglesOneMore}
	\end{figure}
	If $|A| = k\ell_1$ for some $k\in \N$, then $\Sym_{\mathscr{M}}(\mathscr{Q}, A, 2,1)$ (Figure \ref{Case_1_rectanglesOneMore}) is a symmetrization of $\mathscr{M}_{\mathcal{L}}^{-1} (|A|)$ because $\ell_1 + 1 = \ell_2$,
	whence is optimal by Theorem \ref{rectangleProof}.
	Similarly, if $|A| \leq \ell_1$ or $|A|\geq \ell_1(\ell_2-1)$ then $A$ is a symmetrization of $\mathscr{M}_{\mathcal{L}}^{-1} (|A|)$.
	The other direction follows from Lemma \ref{lemmaForAll}.
	
	The second claim now follows immediately because for every optimal colexicographic type set we have a lexicographic type set of the same weight.
\end{proof}

\begin{cor}\label{exactLex}
	Suppose that $\mathscr{M} = \mathscr{M}_{[2]}(\ell_1,\ell_2)$ with $1<\ell_1 < \ell_2-1$ and with a rank increasing and rank constant weight function.
	Also, let $A\subseteq \mathscr{M}$ be a downset.
	\begin{enumerate}
		\item Suppose that $A = \mathscr{M}_{\mathcal{C}}^{-1} (|A|)$, 
		then $A$ is optimal if and only if $|A| \leq \ell_1$ or $|A|\geq \ell_1(\ell_2-1)$.
		\item If $A = \mathscr{M}_{\mathcal{L}}^{-1} (|A|)$ then $A$ is optimal.
	\end{enumerate}
\end{cor}
\begin{proof}
	We prove the first claim.
	So, suppose that $A = \mathscr{M}_{\mathcal{C}}^{-1} (|A|)$.
	The implication, 
	if $|A| \leq \ell_1$ or $|A|\geq \ell_1(\ell_2-1)$ then $A$ is optimal,
	follows from Theorem \ref{rectangleProof},
	since in this case $\mathscr{M}_{\mathcal{C}}^{-1} (|A|)$ and $\mathscr{M}_{\mathcal{L}}^{-1} (|A|)$ are symmetrizations of each other
	and whence have the same weight.
	We handle the other implication.
	For Lemma \ref{lemmaForAll} we get that $|A| \leq \ell_1$ or $|A|\geq \ell_1(\ell_2-1)$ or $|A| = k\ell_1$ for some $k\in \N$.
	Thus, we assume to the contrary that $\ell_1 < |A| < \ell_1(\ell_2-1)$ and $|A| = k\ell_1$ for some $k\in \N$, and show that $A$ is not optimal in this case.
	We handle two cases, one when $\ell_1 < |A| < \ell_1\ell_1$ and another when $\ell_1\ell_1 \leq |A| < \ell_1(\ell_2-1)$.
	
	So, suppose $\ell_1 < |A| < \ell_1\ell_1$. Let (Figure \ref{Case__1_corForAll})
	\begin{align*}
		\mathscr{Q}_1 = \{(a,b)\in \mathscr{M} \bigm | 0\leq a,b \leq \ell_1-1\}.
	\end{align*}
	We consider $\Sym_{\mathscr{M}}(\mathscr{Q}_1, A, 2,1)$ (Figure \ref{Case__1_corForAll}),
	and $\Sym_{\mathscr{M}}(\mathscr{Q}_1, A, 2,1)$ is not optimal by Lemma \ref{case2.3}.
	\begin{figure}
		\centering
		\begin{subfigure}[t]{0.23\textwidth}
			\includegraphics[width=\textwidth]{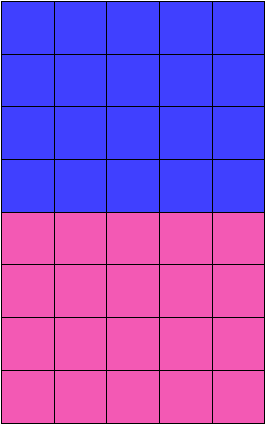}
			\caption{$A$.}
		\end{subfigure}
		\hfill
		\begin{subfigure}[b]{0.1\textwidth}
			\centering
			\raisebox{1.84\hsize}{\includegraphics[width=\textwidth]{section3/2D_Pictures/Case_1_lemmaForAll_arrow_1.pdf}}
		\end{subfigure}
		\hfill
		\begin{subfigure}[t]{0.23\textwidth}
			\includegraphics[width=\textwidth]{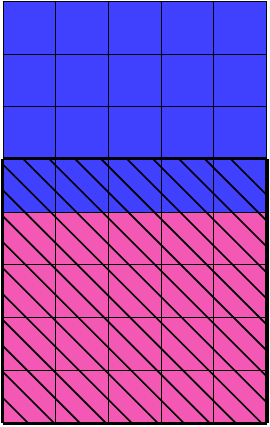}
			\caption{$\mathscr{Q}_1$ (shaded).}
		\end{subfigure}
		\hfill
		\begin{subfigure}[b]{0.18\textwidth}
			\centering
			\raisebox{1.022222222\hsize}{\includegraphics[width=\textwidth]{section3/2D_Pictures/Case_1_lemmaForAll_arrow_2.pdf}}
		\end{subfigure}
		\hfill
		\begin{subfigure}[t]{0.23\textwidth}
			\includegraphics[width=\textwidth]{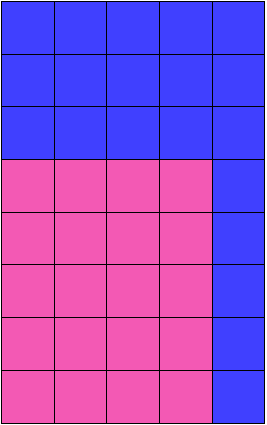}
			\caption{$\Sym_{\mathscr{M}}(\mathscr{Q}_1, A, 2,1)$.}
		\end{subfigure}
		
		\caption{Corollary \ref{exactLex} with $\ell_1 < |A| < \ell_1\ell_1$: $\Sym_{\mathscr{M}}(\mathscr{Q}_2, A, 2,1)$ is not optimal by Lemma \ref{case2.3}.}
		\label{Case__1_corForAll}
	\end{figure}
	
	Next, we handle the case $\ell_1\ell_1 \leq |A| < \ell_1(\ell_2-1)$.
	Let $(\ell_1-1, y)\in \mathscr{M}$ be the first element in $\mathscr{M}_{\{2\}}(\ell_1-1)\setminus A$.
	Also, put (Figure \ref{Case__2_corForAll})
	\begin{align*}
		\mathscr{Q}_2 = \{(a,b)\in \mathscr{M} \bigm | 0\leq a \leq \ell_1-1 \text{ and } y-\ell_1 + 1 \leq b\leq y\}
	\end{align*}
	We consider $\Sym_{\mathscr{M}}(\mathscr{Q}_2, A, 2,1)$ (Figure \ref{Case__2_corForAll}),
	and $\Sym_{\mathscr{M}}(\mathscr{Q}_2, A, 2,1)$ is not optimal by Lemma \ref{case2.4}.
	\begin{figure}
		\centering
		\begin{subfigure}[t]{0.23\textwidth}
			\includegraphics[width=\textwidth]{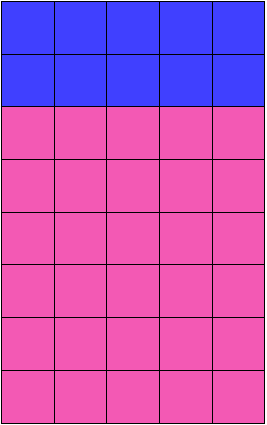}
			\caption{$A$.}
		\end{subfigure}
		\hfill
		\begin{subfigure}[b]{0.1\textwidth}
			\centering
			\raisebox{1.84\hsize}{\includegraphics[width=\textwidth]{section3/2D_Pictures/Case_2_2_lemmaForAll_arrow_1.pdf}}
		\end{subfigure}
		\hfill
		\begin{subfigure}[t]{0.23\textwidth}
			\includegraphics[width=\textwidth]{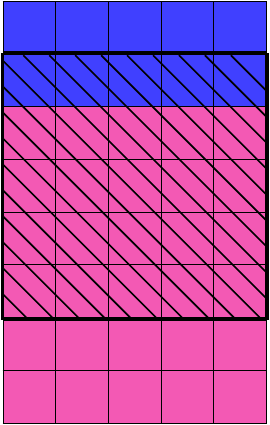}
			\caption{$\mathscr{Q}_2$ (shaded).}
		\end{subfigure}
		\hfill
		\begin{subfigure}[b]{0.18\textwidth}
			\centering
			\centering
			\raisebox{1.022222222\hsize}{\includegraphics[width=\textwidth]{section3/2D_Pictures/Case_2_2_lemmaForAll_arrow_2.pdf}}
		\end{subfigure}
		\hfill
		\begin{subfigure}[t]{0.23\textwidth}
			\includegraphics[width=\textwidth]{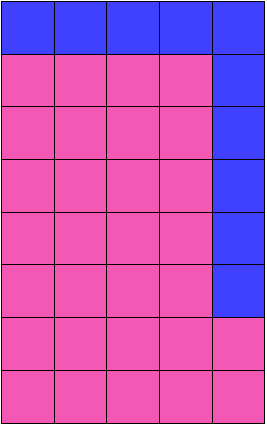}
			\caption{$\Sym_{\mathscr{M}}(\mathscr{Q}_2, A, 2,1)$.}
		\end{subfigure}
		
		\caption{Corollary \ref{exactLex} with $\ell_1\ell_1 \leq |A| < \ell_1(\ell_2-1)$: $\Sym_{\mathscr{M}}(\mathscr{Q}_2, A, 2,1)$ is not optimal by Lemma \ref{case2.4}.}
		\label{Case__2_corForAll}
	\end{figure}
	Therefore, the claim holds.
	
	The second claim now follows immediately because for every optimal colexicographic type set we have a lexicographic type set of the same weight.
\end{proof}

Of course, there are corresponding results when the domination orders switch places.

\begin{cor}\label{rectanglesOneMoreCoLex}
	Suppose that $\mathscr{M} = \mathscr{M}_{[2]}(\ell_1,\ell_2)$ with $1<\ell_2 = \ell_1 -1$ and a rank increasing and rank constant weight function.
	Also, let $A\subseteq \mathscr{M}$ be a downset.
	\begin{enumerate}
		\item Suppose that $A = \mathscr{M}_{\mathcal{L}}^{-1} (|A|)$, 
		then $A$ is optimal if and only if $|A| \leq \ell_2$ or $|A|\geq \ell_2(\ell_1-1)$ or $|A| = k\ell_2$ for some $k\in \N$.
		\item If $A = \mathscr{M}_{\mathcal{C}}^{-1} (|A|)$ then $A$ is optimal.
	\end{enumerate}
\end{cor}

\begin{cor}\label{exactCoLex}
	Suppose that $\mathscr{M} = \mathscr{M}_{[2]}(\ell_1,\ell_2)$ with $1<\ell_2 < \ell_1-1$ and with a rank increasing and rank constant weight function.
	Also, let $A\subseteq \mathscr{M}$ be a downset.
	\begin{enumerate}
		\item Suppose that $A = \mathscr{M}_{\mathcal{L}}^{-1} (|A|)$, 
		then $A$ is optimal if and only if $|A| \leq \ell_2$ or $|A|\geq \ell_2(\ell_1-1)$.
		\item If $A = \mathscr{M}_{\mathcal{C}}^{-1} (|A|)$ then $A$ is optimal.
	\end{enumerate}
\end{cor}

Putting everything together we are able to strengthen Lindsay's edge-isoperimetric inequality in two dimensions again.
This time we get to say that the solution Lindsay found is unique, 
and the same solution remains unique when we use a rank increasing and rank constant weight function.

\begin{cor}\label{nestedSolutionsRectangles}
	Suppose that $\mathscr{M} = \mathscr{M}_{[2]}(\ell_1,\ell_2)$ with a rank increasing and rank constant weight function.
	Then we have:
	\begin{enumerate}
		\item If $\ell_1=\ell_2$ then $\mathcal{L}$ and $\mathcal{C}$ give the only nested solutions for $\mathscr{M}$.
		\item If $\ell_1=1$ or $\ell_2=1$ then there is a unique total order on $\mathscr{M}$,
		and hence $\mathcal{L}$ and $\mathcal{C}$ give the only nested solutions for $\mathscr{M}$.
		\item If $1<\ell_1 < \ell_2$ then $\mathcal{L}$ gives the only nested solutions for $\mathscr{M}$.
		\item If $1<\ell_2 < \ell_1$ then $\mathcal{C}$ gives the only nested solutions for $\mathscr{M}$.
	\end{enumerate}
\end{cor}
\begin{proof}
	Theorem \ref{rectangleProof} and corollaries \ref{squares}, \ref{rectanglesOneMore}, \ref{exactLex}, \ref{rectanglesOneMoreCoLex} 
	and \ref{exactCoLex} prove the claim immediately.
\end{proof}

The results in this section give a complete picture of the behavior of optimal downsets in $\mathscr{M}(\ell_1,\ell_2)$.
However, they say something even deeper about the typical approach to higher dimensional isoperimetric problems.
This is discussed in Section \ref{appliations}.

%auto-ignore
\section{Optimal Downsets In Right Triangles}\label{triangles}

\begin{dfn}[Right Simplices]\label{trapezoids}
	Suppose that we have a multiset lattice $\mathscr{M} = \mathscr{M}_{[d]}(\ell,\dots, \ell)$.
	The \textit{right simplex induced from} $\mathscr{M}$ is the set of all increasing sequences in $\mathscr{M}$, we denote it by
	\begin{align*}
		\mathscr{R}=\mathscr{R}_{[d]}(\ell) = \{(x_1, x_2,\dots, x_d) \in \mathscr{M} \bigm | x_1\leq x_2 \leq \cdots \leq x_d\}.
	\end{align*}
	If $d=2$ then we call $\mathscr{R}$ a \textit{right triangle}.
\end{dfn}

Right triangles are subposets of multiset lattices.
We can talk about the same partial orders and total orders on right triangles that are just induced from the orders on multiset lattices.
For any domination order $\mathcal{D}$, we will write $\mathscr{R}_{\mathcal{D}}$ 
to denote the totally ordered set induced from the corresponding totally ordered multiset lattice $\mathscr{M}_{\mathcal{D}}$.
Often we will write $\mathscr{R}_{\pi}$ for $\mathscr{R}_{D}$, where $\pi \in \mathfrak{S}_d$ is the permutation corresponding to the domination order $\mathcal{D}$.
Similarly, any weight function on a multiset lattice will induce a weight function on the corresponding right triangle.
We can define subproducts like in the case for multiset lattices, for any $S\subseteq [d]$ we define $\mathscr{R}_S = \mathscr{M}_S\cap \mathscr{R}$.
Similarly, for a set of coordinates $S\subseteq [d]$ and a point $x\in \mathscr{R}_{[d]\setminus S}$ we define $\mathscr{R}_S(x) = \mathscr{R}\cap \mathscr{M}_{S}(x)$.

\begin{exm}
	Some examples of the basic properties on right triangles that are induced from multiset lattices can be seen on Figure \ref{trianglesBasics}.
\end{exm}

\begin{figure}
	\centering
	\begin{subfigure}[t]{0.32\textwidth}
		\includegraphics[width=\textwidth]{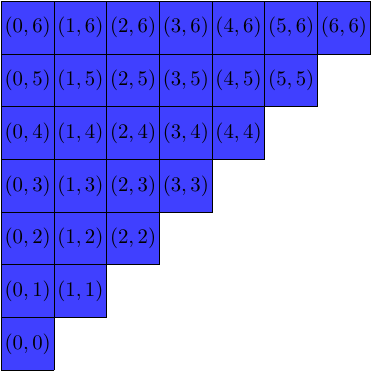}
		\caption{$\mathscr{R}_{[2]}(7)$.}
	\end{subfigure}
	\hfill
	\begin{subfigure}[t]{0.32\textwidth}
		\includegraphics[width=\textwidth]{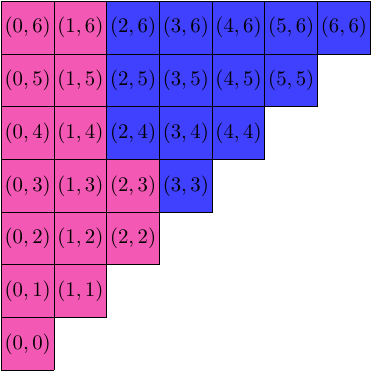}
		\caption{$(\mathscr{R}_{[2]}(7))_{\mathcal{L}}^{-1}[15]$.}
	\end{subfigure}
	\hfill
	\begin{subfigure}[t]{0.32\textwidth}
		\includegraphics[width=\textwidth]{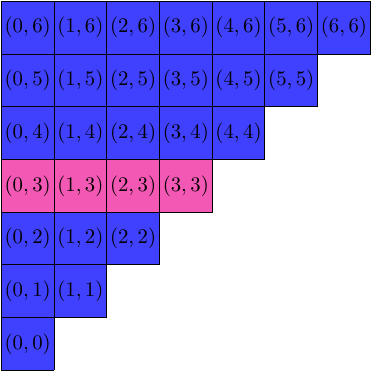}
		\caption{$(\mathscr{R}_{[2]}(7))_{\{1\}}(3)$.}
	\end{subfigure}

	\caption{A right triangle, an initial segment of $\mathcal{L}$, and a subproduct.}
	\label{trianglesBasics}
\end{figure}

Definition \ref{symmetrizationAnyPoset} even allows us to talk about symmetrization in right triangles.
With this in mind, one might ask if similar results to the ones for rectangles hold for right triangles in two dimensions.
It turns out that the main structure theorem is identical to the one for rectangles.
We state and prove this structure theorem after we introduce some tools that we use in the proof.

\begin{dfn}[Diagonal Points]
	Suppose that we have $\mathscr{R}=\mathscr{R}_{[2]}(\ell)$ induced from $\mathscr{M} = \mathscr{M}_{[2]}(\ell, \ell)$, and take a nonempty downset $A\subseteq \mathscr{R}$.
	We define the \textit{diagonal point of} $A$ to be the element $(x,x)\in A$ such that $(x+1,x+1)\not\in A$,
	and we denote it by $d(A)$.
	
	Next, for any $p=(y,y)\in \mathscr{R}$ we define the \textit{diagonal multiset lattice of} $p$,
	to be the largest subposet of $\mathscr{R}$ that contains $p$ as its bottom right corner,
	\begin{align*}
		\mathscr{Q}_p= \{(a,b)\in \mathscr{R} \bigm | 0\leq a \leq y \text{ and } y \leq b \leq \ell -1\}.
	\end{align*}

	Finally we define the \textit{diagonal multiset lattice of} $A$ to be the subposet of $\mathscr{R}$,
	\begin{align*}
		\mathscr{Q}_A = \mathscr{Q}_{d(A)}.
	\end{align*}
	
\end{dfn}

\begin{exm}
	A diagonal point and diagonal multiset lattice can be seen in Figure \ref{trianglesTricks}.
\end{exm}

\begin{figure}
	\centering
	\begin{subfigure}[t]{0.4\textwidth}
		\includegraphics[width=\textwidth]{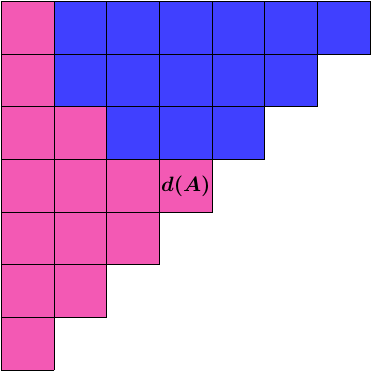}
		\caption{$A$ with $d(A)$.}
	\end{subfigure}
	\hfill
	\begin{subfigure}[t]{0.4\textwidth}
		\includegraphics[width=\textwidth]{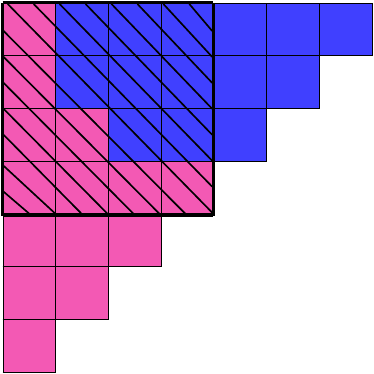}
		\caption{$A$ with $\mathscr{Q}_A$.}
	\end{subfigure}
	
	\caption{A downset $A$ with $d(A)$ and $\mathscr{Q}_A$.}\label{trianglesTricks}
\end{figure}

\begin{thm}\label{triangleProof}
	Suppose that $\mathscr{R} = \mathscr{R}_{[2]}(\ell)$ 
	is a right triangle with a rank increasing and rank constant weight function.
	If $A\subseteq \mathscr{R}$ is an optimal downset then one of the following statements must hold:
	\begin{enumerate}
		\item $A$ is an initial segment of a domination order: that is, there exists $\pi\in \mathfrak{S}_2$ such that $A=\mathscr{R}_{\pi}^{-1}[|A|]$.
		\item $A$ is a symmetrization (using a packed poset) of an initial segment of a domination order: that is,
		there exist $\pi\in \mathfrak{S}_2$, a packed poset $\mathscr{Q}\subseteq \mathscr{R}$ and coordinates $c_1,c_2\in [2]$,
		such that $A = \Sym_{\mathscr{R}}(\mathscr{Q},\mathscr{R}_{\pi}^{-1}[|A|], c_1,c_2 )$.
	\end{enumerate}
\end{thm}
\begin{proof}
	We assume that $A \neq \emptyset$ and $A \neq \mathscr{R}$, since in these cases $A$ is an initial segment of both $\mathcal{L}$ and $\mathcal{C}$.
	Let $d(A)=(x,x)$, whence $\ell \geq 2$ and $x \neq \ell - 1$ since $A \neq \mathscr{R}$.
	If $x=0$ then $A\subseteq \mathscr{R}_{\{2\}}(0)$ and we have that $A$ is an initial segment on $\mathcal{L}$.
	If $x=\ell -2$ then $\mathscr{R}\setminus A \subseteq \mathscr{R}_{\{1\}}(\ell -1)$, hence $A$ is an initial segment of $\mathcal{C}$.
	Thus, we suppose that $\ell \geq 4$ and $1\leq x \leq \ell - 3$.
	
	Let $\mathscr{Q}_A \iso \mathscr{M}(\ell_1, \ell_2)$.
	We are going to handle the theorem by proving things about $\mathscr{Q}_A$.
	In order for $A$ to be optimal in $\mathscr{R}$, we need to have that $\mathscr{Q}_A\cap A$ is optimal in $\mathscr{Q}_A$.
	However, Theorem \ref{rectangleProof} and its corollaries restrict the options for $\mathscr{Q}_A\cap A$.
	
	\begin{lem}\label{triangleLexLemma}
		If $\mathscr{Q}_A\cap A$ is a lexicographic type set in $\mathscr{Q}_A$ then the claim holds.
	\end{lem}
	\begin{proof}
		First suppose that $(x,x+1)\in A$.
		Then $\mathscr{Q}_A\cap A$ must fall under \ref{type1Cor} or \ref{type2.2} of Corollary \ref{fullRectangleProof}, since type \ref{type2.1} requires $(x,x+1)\not\in A$.
		If $\mathscr{Q}_A\cap A$ falls under \ref{type1Cor} then $A$ is an initial segment of $\mathcal{L}$ (Figure \ref{triangleLexCase1A}).
		If $\mathscr{Q}_A\cap A$ falls under \ref{type2.2} then $A$ is a symmetrization of an initial segment of $\mathcal{L}$ (Figure \ref{triangleLexCase1B}).
		\begin{figure}
			\centering
			\begin{subfigure}[t]{0.4\textwidth}
				\includegraphics[width=\textwidth]{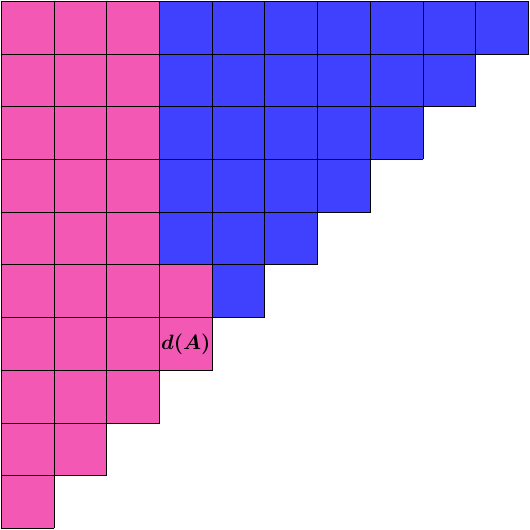}
				\caption{An initial segment of $\mathcal{L}$.}
				\label{triangleLexCase1A}
			\end{subfigure}
			\hfill
			\begin{subfigure}[t]{0.4\textwidth}
				\includegraphics[width=\textwidth]{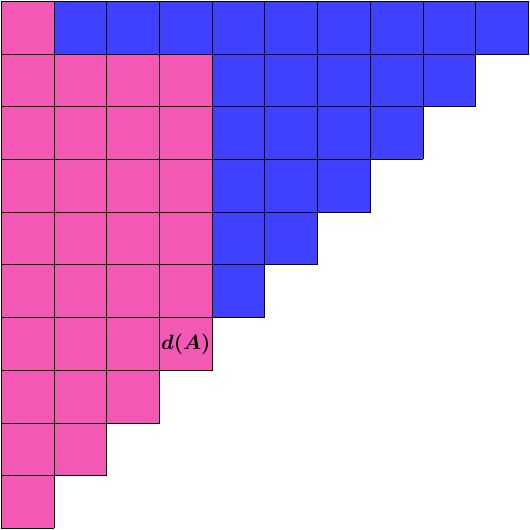}
				\caption{A symmetrization of an initial segment of $\mathcal{L}$.}
				\label{triangleLexCase1B}
			\end{subfigure}
			
			\caption{Lemma \ref{triangleLexLemma} with $(x,x+1)\in A$.}
			\label{triangleLexCase1}
		\end{figure}
	
		So, suppose that $(x,x+1)\not\in A$.
		If $|\mathscr{Q}_A\cap A|= (\ell_1-1)\ell_2 + 1$ then $\mathscr{Q}_A\cap A$ is an initial segment of $\mathcal{L}$, whence $A$ is an initial segment of $\mathcal{L}$.
		Assume that $|\mathscr{Q}_A\cap A|< (\ell_1-1)\ell_2 + 1$.
		Then $(x-1, \ell_2-1)\not \in A$, but this forces $\mathscr{Q}_A\cap A$ to fall under \ref{type2.1} of Corollary \ref{fullRectangleProof}, since $(x,x)\in \mathscr{Q}_A\cap A$.
		
		Let $n\in \N$ be the smallest integer such that $(n,x+1)\not\in A$.
		If $n=0$ then $A$ is an initial segment of $\mathcal{C}$.
		If $n=x-1$ then $A$ is a symmetrization of an initial segment of $\mathcal{L}$ (Figure \ref{triangleLexCase2}).
		\begin{figure}
			\centering
			\begin{subfigure}[t]{0.4\textwidth}
				\includegraphics[width=\textwidth]{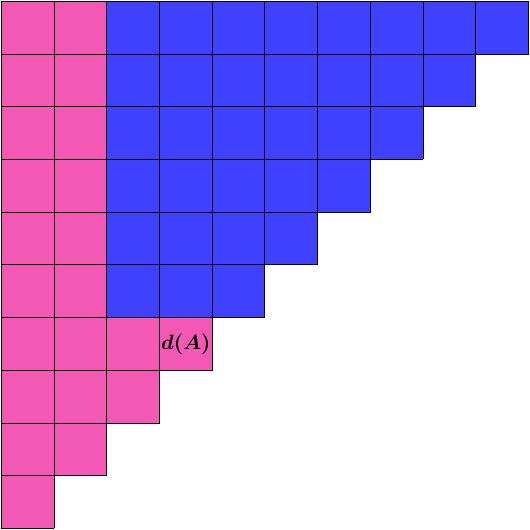}
				\caption{$A$.}
			\end{subfigure}
			\hfill
			\begin{subfigure}[b]{0.18\textwidth}
				\centering
				\raisebox{1.111111111\hsize}{{\includegraphics[width=\textwidth]{section3/2D_Pictures/Case_1_rectanglesOneMore_arrow_2.pdf}}}
			\end{subfigure}
			\hfill
			\begin{subfigure}[t]{0.4\textwidth}
				\includegraphics[width=\textwidth]{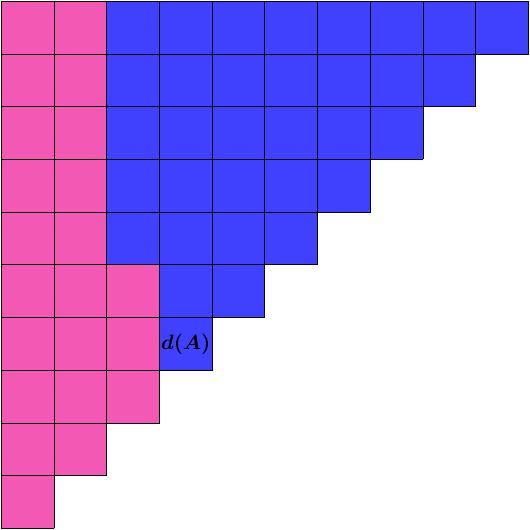}
				\caption{$\Sym_{\mathscr{M}}(\mathscr{Q}, A, 2,1)$.}
			\end{subfigure}
			
			\caption{Lemma \ref{triangleLexLemma} with $(x,x+1)\not\in A$ and $n=x-1$: $A$ is a symmetrization of an initial segment of $\mathcal{L}$ by a packed poset $\mathscr{Q}$.}
			\label{triangleLexCase2}
		\end{figure}
	
		Thus, we assume that $1\leq n \leq x-2$.
		Also, note that $\ell_2 \geq 3>1$ because $x \leq \ell-3$.
		However, now we must have that $\ell_2 \geq \ell_1 - 1$ by Corollary \ref{rectanglesOneMoreCoLex} and Corollary \ref{exactCoLex}.
		If $\ell_2=\ell_1-1$ and $n=1$ then $A$ is a symmetrization of an initial segment of $\mathcal{C}$ (Figure \ref{triangleLexCase3}).
		\begin{figure}
			\centering
			\begin{subfigure}[t]{0.4\textwidth}
				\includegraphics[width=\textwidth]{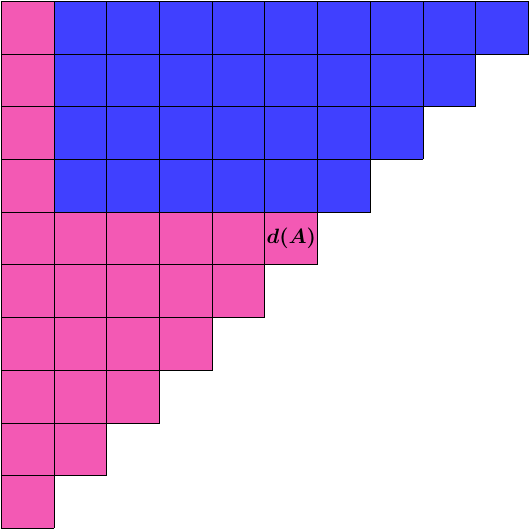}
				\caption{$A$.}
			\end{subfigure}
			\hfill
			\begin{subfigure}[b]{0.18\textwidth}
				\centering
				\raisebox{1.111111111\hsize}{{\includegraphics[width=\textwidth]{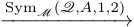}}}
			\end{subfigure}
			\hfill
			\begin{subfigure}[t]{0.4\textwidth}
				\includegraphics[width=\textwidth]{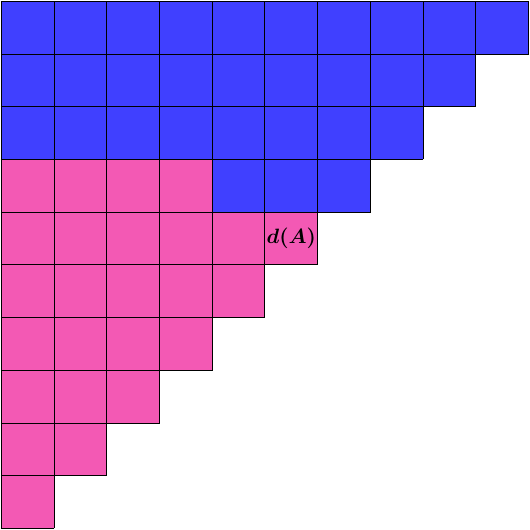}
				\caption{$\Sym_{\mathscr{M}}(\mathscr{Q}, A, 1,2)$.}
			\end{subfigure}

			\caption{Lemma \ref{triangleLexLemma} with $(x,x+1)\not\in A$, $\ell_2=\ell_1-1$ and $n=1$: $A$ is a symmetrization of an initial segment of $\mathcal{C}$ by a packed poset $\mathscr{Q}$.}
			\label{triangleLexCase3}
		\end{figure}
	
		It turns out that the above cases are the only ones possible.
		Assume to the contrary that $\ell_2>\ell_1-1$ or $1<n\leq x-2$.
		Then $A$ is not optimal because we can consider a symmetrization of $A$ by a packed poset $\mathscr{Q}$, 
		call this set $B = \Sym_{\mathscr{M}}(\mathscr{Q}, A, 2,1)$ (Figure \ref{triangleLexCase4}), 
		then $\mathscr{Q}_B\cap B$ is not optimal by Corollary \ref{fullRectangleProof}, whence $A$ is not optimal.
		\begin{figure}
			\centering
			\begin{subfigure}[t]{0.4\textwidth}
				\includegraphics[width=\textwidth]{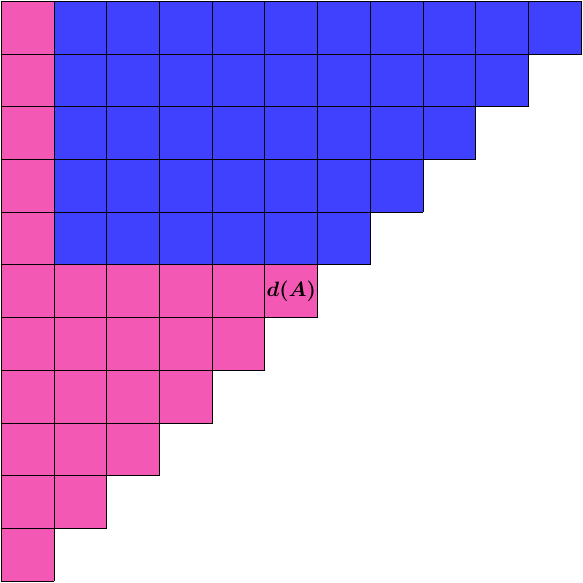}
				\caption{$A$.}
			\end{subfigure}
			\hfill
			\begin{subfigure}[b]{0.1\textwidth}
				\centering
				\raisebox{2\hsize}{\includegraphics[width=\textwidth]{section2/2D_Pictures/Case_1_2_rectangleProof_arrow.pdf}}
			\end{subfigure}
			\hfill
			\begin{subfigure}[t]{0.4\textwidth}
				\includegraphics[width=\textwidth]{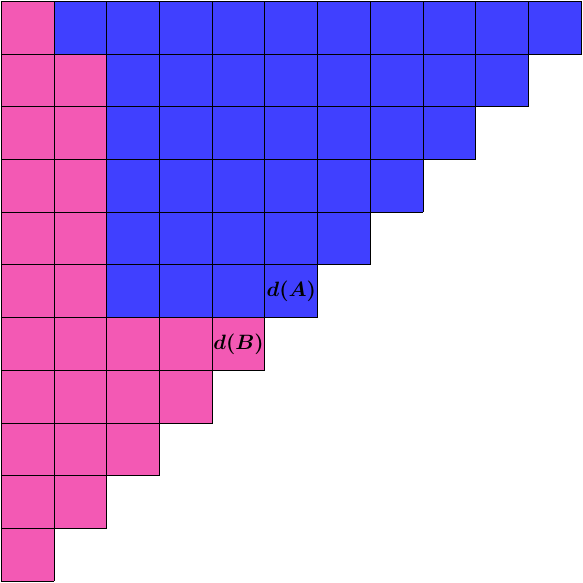}
				\caption{$B=\Sym_{\mathscr{M}}(\mathscr{Q}, A, 1,2)$.}
			\end{subfigure}

			\caption{Lemma \ref{triangleLexLemma} with $(x,x+1)\not\in A$, $\ell_2>\ell_1-1$ or $1<n\leq x-2$: 
				$A$ is not optimal because a symmetrization of it is not optimal by Corollary \ref{fullRectangleProof}.}\label{triangleLexCase4}
		\end{figure}
	\end{proof}

	\begin{lem}\label{triangleColexLemma}
		If $\mathscr{Q}_A\cap A$ is a colexicographic type set in $\mathscr{Q}_A$ then the claim holds.
	\end{lem}
	\begin{proof}
		If $|\mathscr{Q}_A\cap A| = \ell_1$ then $A$ is a an initial segment of $\mathcal{C}$.
		Similarly, if $|\mathscr{Q}_A\cap A| = \ell_1\ell_2$ then $A$ is an initial segment of $\mathcal{L}$.
		So, suppose that $\ell_1 < |\mathscr{Q}_A\cap A| < \ell_1\ell_2$ for the rest of this proof.
		
		First, suppose $\ell_1 < \ell_2-1$.
		Then $|\mathscr{Q}_A\cap A| \leq \ell_1$ or $|\mathscr{Q}_A\cap A|\geq \ell_1(\ell_2-1)$ by Corollary \ref{exactLex}.
		Thus, $|\mathscr{Q}_A\cap A|\geq \ell_1(\ell_2-1)$ and $A$ is a symmetrization of an initial segment of $\mathcal{L}$, since $\ell_1 < \ell_2-1$.
		Hence, we assume that $\ell_1 \geq \ell_2 - 1$ for the rest of this proof.
		Note that $\ell_2 \geq 3$ because $x \leq \ell -3$, whence $\ell_1\geq 2$.
		
		Next, suppose that $|\mathscr{Q}_A\cap A|\leq 2\ell_1$.
		First, we handle the edge case $\ell_1=2$.
		Then $2 \geq \ell_2 -1$, which gives $\ell_2\leq 3$.
		However, this gives $\ell_2=3$, since we have $\ell_2\geq 3$.
		But now Corollary \ref{rectanglesOneMore} implies that $|\mathscr{Q}_A\cap A| = 4$.
		This gives only two options for $\mathscr{Q}_A\cap A$ based on if $(x,x+1)\in A$ (Figure \ref{triangleCoLexCase1}),
		for one of them $A$ is an initial segment of $\mathcal{C}$ (Figure \ref{triangleCoLexCase1A}) when $(x,x+1)\in A$,
		and in the other $A$ is initial segment of $\mathcal{L}$ (Figure \ref{triangleCoLexCase1B}) when $(x,x+1)\not \in A$.
		\begin{figure}
			\centering
			\begin{subfigure}[t]{0.3\textwidth}
				\includegraphics[width=\textwidth]{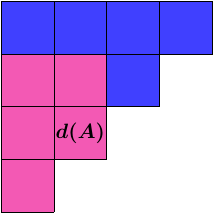}
				\caption{An initial segment of $\mathcal{C}$.}
				\label{triangleCoLexCase1A}
			\end{subfigure}
			\hspace{2cm}
			\begin{subfigure}[t]{0.3\textwidth}
				\includegraphics[width=\textwidth]{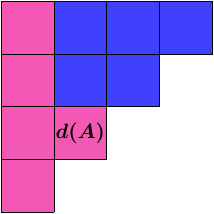}
				\caption{An initial segment of $\mathcal{L}$.}
				\label{triangleCoLexCase1B}
			\end{subfigure}
			
			\caption{Lemma \ref{triangleColexLemma} with $|\mathscr{Q}_A\cap A|\leq 2\ell_1$ and $\ell_1 = 2$.}
			\label{triangleCoLexCase1}
		\end{figure}
		So, we assume that $\ell_1 \geq 3$.
		Then $\mathscr{Q}_A\cap A$ must be of type \ref{type1Cor} or \ref{type2.1} in Corollary \ref{fullRectangleProof}, 
		notice that we can never have type \ref{type2.2} because $\ell_1 \geq \ell_2 - 1$, $\ell_1 \geq 3$ and $(x,x)\in \mathscr{Q}_A\cap A$.
		If $\mathscr{Q}_A\cap A$ is of type \ref{type1Cor} then $A$ is an initial segment of $\mathcal{C}$.
		If $\mathscr{Q}_A\cap A$ is of type \ref{type2.1} then $A$ is a symmetrization of an initial segment of $\mathcal{C}$.
		Hence, we assume that $|\mathscr{Q}_A\cap A|> 2\ell_1$ for the rest of this proof.
		
		Next, we handle the case $|\mathscr{Q}_A\cap A|= 3\ell_1$.
		Well, $\mathscr{Q}_A\cap A$ has to be one of the 3 types of colexicographic sets in Corollary \ref{fullRectangleProof}.
		It can't be of type \ref{type2.1} because $\ell_1\geq \ell_2-1$ and $|\mathscr{Q}_A\cap A|= 3\ell_1 > 2\ell_1$.
		If it is of type \ref{type1Cor} then $A$ is a symmetrization of an initial segment of $\mathcal{C}$ (Figure \ref{triangleCoLexCase2}).
		\begin{figure}
			\centering
			\begin{subfigure}[t]{0.4\textwidth}
				\includegraphics[width=\textwidth]{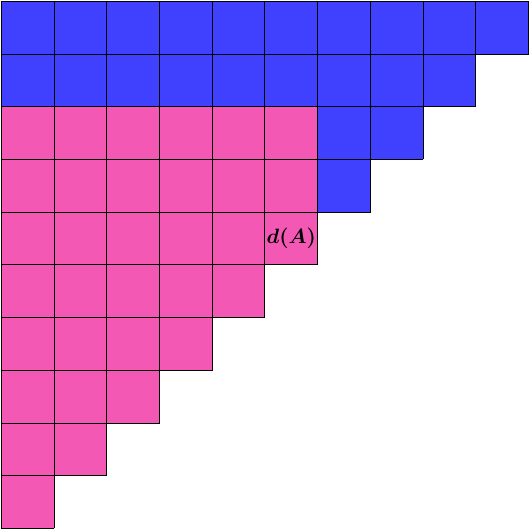}
				\caption{$A$.}
			\end{subfigure}
			\hfill
			\begin{subfigure}[b]{0.18\textwidth}
				\centering
				\raisebox{1.111111111\hsize}{{\includegraphics[width=\textwidth]{section4/2D_Pictures/triangleLexCase3_arrow.pdf}}}
			\end{subfigure}
			\hfill
			\begin{subfigure}[t]{0.4\textwidth}
				\includegraphics[width=\textwidth]{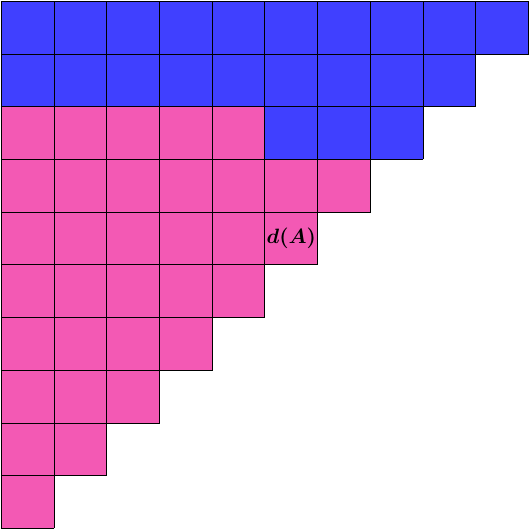}
				\caption{$\Sym_{\mathscr{M}}(\mathscr{Q}, A, 1,2)$.}
			\end{subfigure}

			\caption{Lemma \ref{triangleColexLemma} with $|\mathscr{Q}_A\cap A|= 3\ell_1$ and $\mathscr{Q}_A\cap A$ of type \ref{type1Cor}: $A$ is a symmetrization of an initial segment of $\mathcal{C}$ by a packed poset $\mathscr{Q}$.}
			\label{triangleCoLexCase2}
		\end{figure}
		So we assume that $\mathscr{Q}_A\cap A$ is of type \ref{type2.2}.
		Then $\ell_1 = 3$, which implies that $\ell_2 \leq \ell_1+1 = 4$.
		We can't have $\ell_2 = 3$ because $|\mathscr{Q}_A\cap A| < \ell_1\ell_2$, whence $\ell_2 = 4$.
		However, in this case $A$ is an initial segment of $\mathcal{L}$ (Figure \ref{triangleCoLexCase3}).
		\begin{figure}
			\centering
			\includegraphics[width=0.3\textwidth]{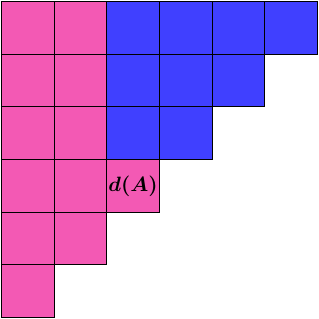}
			\caption{Lemma \ref{triangleColexLemma} with $|\mathscr{Q}_A\cap A|= 3\ell_1$ and $\mathscr{Q}_A\cap A$ of type \ref{type2.2}: $A$ is an initial segment of $\mathcal{L}$.}
			\label{triangleCoLexCase3}
		\end{figure}
		
		It turns out that the above cases are the only ones possible.
		Assume to the contrary that $|\mathscr{Q}_A\cap A|> 2\ell_1$ and $|\mathscr{Q}_A\cap A|\neq 3\ell_1$.
		We show that if $\mathscr{Q}_A\cap A$ is a colexicographic type set then it is not optimal, which will give us a contradiction.
		So, without loss of generality, we assume that $\mathscr{Q}_A\cap A$ is an initial segment of $\mathcal{C}$, as any symmetrization has the same weight.
		Let $y\in \N$ be the last integer such that $\mathscr{R}_{\{1\}}(y)\cap A \neq \emptyset$ and put $H=\mathscr{R}_{\{1\}}(y)\cap \mathscr{Q}_{A}$ (Figure \ref{triangleCoLexCase_H_and_V}).
		Also, we define $V = \{(a,b)\in \mathscr{R}(\ell) \bigm |  a = x+1 \text{ and } x+1\leq b < y\}$ (Figure \ref{triangleCoLexCase_H_and_V}).
		\begin{figure}
			\centering
			\begin{subfigure}[t]{0.4\textwidth}
				\includegraphics[width=\textwidth]{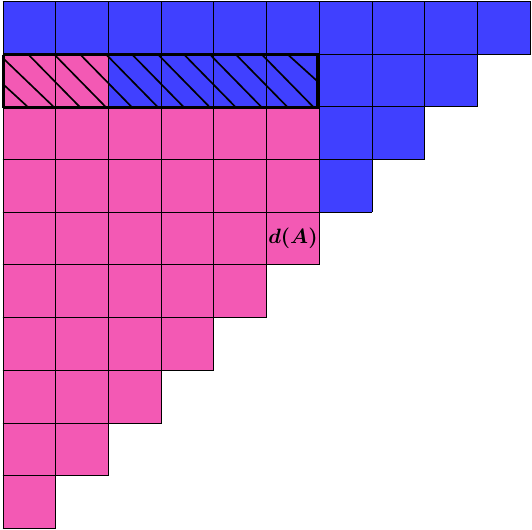}
				\caption{The set $H$ (shaded).}
			\end{subfigure}
			\hfill
			\begin{subfigure}[t]{0.4\textwidth}
				\includegraphics[width=\textwidth]{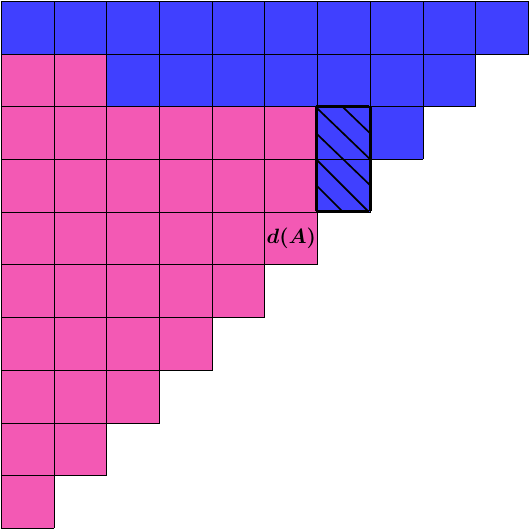}
				\caption{The set $V$ (shaded).}
			\end{subfigure}
			
			\caption{Lemma \ref{triangleColexLemma} with $|\mathscr{Q}_A\cap A|> 2\ell_1$ and $|\mathscr{Q}_A\cap A|\neq 3\ell_1$.}
			\label{triangleCoLexCase_H_and_V}
		\end{figure}
		Note that $|H| = \ell_1$ and $|V| = y-x-1\leq \ell_2 - 1 -1 \leq \ell_1 -1 < \ell_1 =|H|$.
		Define the packed poset (Figure \ref{triangleCoLexCase_Q})
		\begin{align*}
			\mathscr{Q} = \{(a,b) \in \mathscr{M}_{[2]}(\ell, \ell) \bigm | 0 \leq a \leq x+1 \text{ and } y-x-1\leq b\leq y\}.
		\end{align*}
		\begin{figure}
			\centering
			\includegraphics[width=0.5\textwidth]{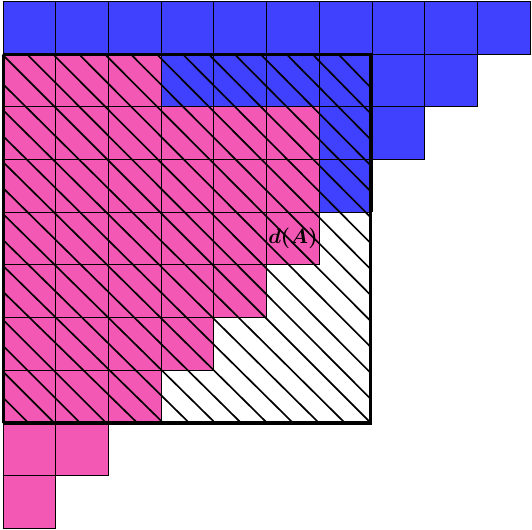}
			\caption{Lemma \ref{triangleColexLemma} with $|\mathscr{Q}_A\cap A|> 2\ell_1$ and $|\mathscr{Q}_A\cap A|\neq 3\ell_1$: 
				The packed poset $\mathscr{Q}$.}\label{triangleCoLexCase_Q}
		\end{figure}
		Then $H,V\subseteq \mathscr{Q}$, and every element of $V$ is a reflection ( in $\mathscr{Q}$) of and element from $H$, since $|V|< |H|$.
		Thus, there is a weight preserving injection from $V$ to $H$.
		First, we suppose that $|\mathscr{Q}_A\cap A|\neq k\ell_1$ for all $k\in \N$.
		Then $A$ is not optimal by a reflect-push method (Figure \ref{triangleCoLexCaseContra1}) that moves $n$ elements with $n = \min\{|H\cap A|, |V|\}$,
		such that the symmetrization is happening in $\mathscr{Q}$ with respect to $(1,2)$,
		and the pushing is upwards.
		\begin{figure}
			\centering
			\begin{subfigure}[t]{0.4\textwidth}
				\includegraphics[width=\textwidth]{section4/2D_Pictures/cube_diagram_10x10_lemma_2_Q.pdf}
				\caption{$A$.}
			\end{subfigure}
			\hfill
			\begin{subfigure}[b]{0.18\textwidth}
				\centering
				\raisebox{1.111111111\hsize}{{\includegraphics[width=\textwidth]{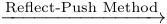}}}
			\end{subfigure}
			\hfill
			\begin{subfigure}[t]{0.4\textwidth}
				\includegraphics[width=\textwidth]{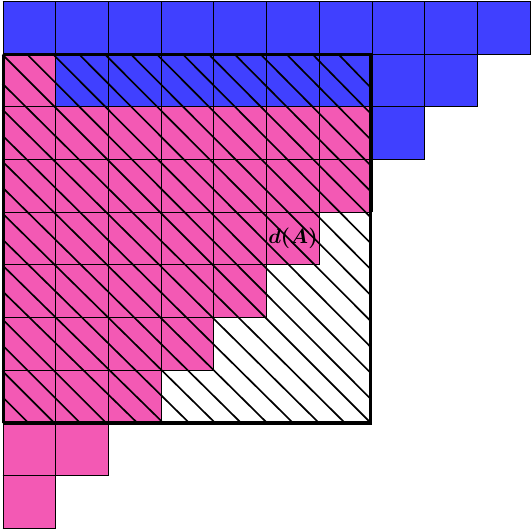}
				\caption{A set with larger weight than $A$.}
			\end{subfigure}
			\caption{Lemma \ref{triangleColexLemma} with $|\mathscr{Q}_A\cap A|> 2\ell_1$ and $|\mathscr{Q}_A\cap A|\neq k\ell_1$: $A$ is not optimal by a reflect-push method.}
			\label{triangleCoLexCaseContra1}
		\end{figure}
		So, $|\mathscr{Q}_A\cap A|= k\ell_1$ for some $k\in \N$ and $k\neq 3$.
		However, now we can preform a symmetrization that moves $|V|$ elements in $\mathscr{Q}$ with respect to $(1,2)$,
		and the resulting set $B$ is not optimal, since it satisfies the conditions from the previous argument.
		Therefore, we get a contradiction in both cases and the claim holds.
	\end{proof}

	Putting Lemma \ref{triangleLexLemma} and Lemma \ref{triangleColexLemma} together, we get that the theorem holds.
\end{proof}

If one follow the proof of Theorem \ref{triangleProof} then a much more specific statement comes out.
This is similar to Corollary \ref{fullRectangleProof} where the packed posets are explicitly stated.
This is an easy exercise and is left to the reader.
From Theorem \ref{triangleProof} we can conclude that the domination orders are best possible.
This is captured in the following corollary.

\begin{cor}\label{triangleCor}
	Suppose that $\mathscr{R} = \mathscr{R}_{[2]}(\ell)$ 
	is a right triangle with a rank increasing and rank constant weight function.
	If $A\subseteq \mathscr{R}$ is a downset then
	\begin{align*}
		\wt(A) \leq \max \{\wt(\mathscr{R}_{\mathcal{L}}^{-1}[|A|]), \wt(\mathscr{R}_{\mathcal{C}}^{-1}[|A|])\}
	\end{align*} 
\end{cor}
\begin{proof}
	Follows right away from Theorem \ref{triangleProof}.
\end{proof}

It turns out that we can prove Corollary \ref{triangleCor} without the machinery of Theorem \ref{triangleProof}.
In fact, we can state an even stronger statement.

\begin{cor}
	Suppose that $\mathscr{R} = \mathscr{R}_{[2]}(\ell)$ 
	is a right triangle with a rank increasing and rank constant weight function.
	Let $A\subseteq \mathscr{R}$ be a downset with $d(A)=(x,x)$.
	One has:
	\begin{enumerate}
		\item if $x\leq \ell/2$ then $\wt(A) \leq \wt(\mathscr{R}_{\mathcal{L}}^{-1}[|A|])$.
		\item if $x\geq \ell/2$ then $\wt(A) \leq \wt(\mathscr{R}_{\mathcal{C}}^{-1}[|A|])$.
	\end{enumerate}
	Therefore,
	\begin{align*}
		\wt(A) \leq \max \{\wt(\mathscr{R}_{\mathcal{L}}^{-1}[|A|]), \wt(\mathscr{R}_{\mathcal{C}}^{-1}[|A|])\}.
	\end{align*} 
\end{cor}
\begin{proof}
	We prove the first claim; the proof for the second one is similar.
	So suppose that $x\leq \ell/2$ and let $A_0 = A$.
	Then suppose that $n\geq 1$ and that for all $n'<n$ the set $A_{n'}\subseteq \mathscr{R}$ is defined.
	Let $\mathscr{Q}=\mathscr{Q}_{A_{n-1}}$.
	We define
	\begin{align*}
		A_n = (A_{n-1}\setminus(\mathscr{Q}\cap A_{n-1}))\cup \mathscr{Q}_{\mathcal{L}}^{-1}[|\mathscr{Q}\cap A_{n-1}|].
	\end{align*}
	Thus, we consider the sequence $(A_n)_{n=0}^\infty$.
	For any $n\geq 1$ we have that $\wt(A_{n-1}) \leq \wt(A_{n})$, since $x\leq \ell/2$.
	The sequence is eventually constant because $A_{n-1} \neq A_{n}$ iff the diagonal point of $A_{n}$ has smaller entries than the entries of the diagonal point of $A_{n-1}$.
	It is easily seen that the constant at the tail of this sequence if an initial segment of $\mathcal{L}$.
	Hence, the first claim is proved.
	The second claim follows by a similar argument.
	The last claim follows by combining the first two claims.
\end{proof}

%auto-ignore
\section{Applications Of The Main Results}\label{appliations}
We apply the results obtained in the previous sections to several extremal problems.
For this we formally state the classical two edge-isoperimetric problems.
All graphs in this section are simple.

\begin{dfn}[Induced Edges]\label{inducedEdges}
	For a graph $G=(V_G,E_G)$ and $A\subseteq V_G$ we define $I_G(A)$ to be the \textit{set of edges induced by} $A$, that is,
	\begin{align*}
		I_G(A) = \{\{x,y\}\in E_G \bigm | x,y\in A\}.
	\end{align*}
	For $m\in \N$ we define $I_G(m)$ to be \textit{the maximum number of edges induced by a set} $A\subseteq V_G$ \textit{of size} $m$, that is,
	\begin{align*}
		I_G(m) = \max_{\substack{A\subseteq V_G\\|A|=m}} |I_G(A)|.
	\end{align*} 
\end{dfn}

\begin{prb}[Induced Edges Problem]\label{inducedEdgesProblem}
	Given a graph $G=(V_G,E_G)$, for each $m\in [|V_G|+1]_0$ find a set $A\subseteq V_G$ such that $|A|= m$ and $|I_G(A)| = I_G(|A|)$.
	We call such a set $A$ \textit{optimal for the induced edges problem}.
\end{prb}

\begin{dfn}[Boundary Edges]\label{boundaryEdges}
	For a graph $G=(V_G,E_G)$ and $A\subseteq V_G$ we define $\Theta_G(A)$ to be the \textit{set of boundary edges of} $A$, that is,
	\begin{align*}
		\Theta_G(A) = \{\{x,y\}\in E_G \bigm | x\in A \text{ and } y\not\in A\}.
	\end{align*}
	For $m\in \N$ we define $\Theta_G(m)$ to be \textit{the minimum number of boundary edges of a set} $A\subseteq V_G$ \textit{of size} $m$, that is,
	\begin{align*}
		\Theta_G(m) = \min_{\substack{A\subseteq V_G\\|A|=m}} |\Theta_G(A)|.
	\end{align*} 
\end{dfn}

\begin{prb}[Boundary Edges Problem]\label{boundaryEdgesProblem}
	Given a graph $G=(V_G,E_G)$, for each $m\in [|V_G|+1]_0$ find a set $A\subseteq V_G$ such that $|A|= m$ and $|\Theta_G(A)| = \Theta_G(|A|)$.
	We call such a set $A$ \textit{optimal for the boundary edges problem}.
\end{prb}

In the above definitions and problems we omit the index $G$ when the graph is clear from context.
Problems \ref{inducedEdgesProblem} and \ref{boundaryEdgesProblem} are equivalent when the graph is regular.
The following trivial result shows us this.
This lemma is a folklore result, but a proof of it can be found in \cite{HarperBook}.

\begin{lem}\label{edgeIsoEquiv}
	If $G=(V,E)$ is a regular graph with degree $k$ then for any $A\subseteq V$ we have
	\begin{align*}
		2|I(A)| + |\Theta(A)| = k|A|.
	\end{align*}
\end{lem}

\subsection{The Ahlswede-Katona Problem}
The Ahlswede-Katona Problem is an extremal problem concerning the number of pairs of adjacent edges in a graph.

\begin{dfn}
	Let $n,m\in \N$ and by $\mathscr{G}(n,m)$ denote the set of all graphs on $n$ vertices and $m$ edges.
	For $G=(V,E)\in \mathscr{G}(n,m)$ we denote the \textit{set of pairs of adjacent edges in} $G$ by
	\begin{align*}
		P(G) = \left\lbrace \{e_1,e_2\} \in \binom{E}{2}\bigm | |e_1\cap e_2| = 1 \right\rbrace .
	\end{align*}
	For any $m\in \left[\binom{n}{2} +1 \right]_0$ we define
	\begin{align*}
		P_n(m) = \max_{\substack{G\in \mathscr{G}(n,m)}} |P(G)|.
	\end{align*}
\end{dfn}

\begin{prb}[The Ahlswede-Katona Problem]\label{AKProblem}
	For any $m\in \left[\binom{n}{2} +1 \right]_0$ find $G \in \mathscr{G}(n,m)$ such that $P(m) = |P(G)|$.
	We call such a graph $G$ \textit{optimal in} $\mathscr{G}(n,m)$.
\end{prb}

Problem \ref{AKProblem} is a special case of problems \ref{inducedEdgesProblem} and \ref{boundaryEdgesProblem}.
This is captured in the following lemma.

\begin{lem}
	Let $n,m\in \N$ and consider a graph $G=(V_G,E_G) \in \mathscr{G}(n,m)$.
	Let $J=(V_J, E_J)$ denote the Johnson graph $J(n,2)$.
	Then $P(G) = I_J(E_G)$.
	By Lemma \ref{edgeIsoEquiv}, $G$ is optimal in $\mathscr{G}(n,m)$ iff $E_G$ is optimal for $I_J$ iff $E_G$ is optimal for $\Theta_J$.
\end{lem}

The only thing to do now is to reduce the edge-isoperimetric problem on $J(n,2)$ to a maximum weight downset problem on one of the posets that we studied in the previous sections.
The following lemma summarizes several strands of work. 
It is an easy exercise on shifting/stabilization and proofs of it can already been found in the literature \cite{ACHyper, HarperBook}.

\begin{lem}
	Consider the Johnson graph $J=(V_J,E_J)=J(n,2)$.
	For any $v\in V_J$ we can write $v=\{a_1,a_2\}$ with $a_1< a_2$,
	and hence we have a bijection $\sigma: V_J \rightarrow \mathscr{R}_{[2]}(n-1)$ defined by $\sigma(v) = (a_1-1, a_2-2)$.
	Giving $\mathscr{R}_{[2]}(n-1)$ with the standard weight function,
	for every $A\subseteq \mathscr{R}_{[2]}(n-1)$ there is a downset $B\subseteq \mathscr{R}_{[2]}(n-1)$ such that:
	\begin{enumerate}
		\item $|B| = |A|$.
		\item $|I_J(\sigma^{-1}(B))| \geq |I_J(\sigma^{-1}(A))|$.
		\item $|I_J(\sigma^{-1}(B))| = \wt(B)$.
	\end{enumerate}
	Thus, for any $n,m$ we can find a $G=(V,E)\in \mathscr{G}(n,m)$ with $E$ optimal for $I_J$, 
	and such that $\sigma(E)$ is an optimal downset in $\mathscr{R}_{[2]} (n-1)$ with the standard weight function.
\end{lem}

So, we know that Problem \ref{AKProblem} is just a special case of Problem \ref{maximumWeightIdealProblem} on $\mathscr{R}_{[2]}(\ell)$,
and in particular the results in section \ref{triangles} apply to Problem \ref{AKProblem}.
The result of Ahlswede and Katona \cite{Ahlswede1978GraphsWM} is the corresponding special case of Corollary \ref{triangleCor}.
The results in section \ref{triangles} are much stronger and describe the behavior of all optimal downsets.
We should mention that Theorem 3 in \cite{Ahlswede1978GraphsWM} compared initial segments of $\mathcal{L}$ and $\mathcal{C}$ in more detail.
The proof of this theorem is somewhat complicated.
In the near future, the authors of the current paper hope to apply the reflect-push method to strengthen these results and provide a combinatorial and geometric proof. 

The result of Ahlswede and Katona was recently generalized in another direction by Keough and Radcliffe in \cite{KeoughLauren2016Gwtf}.
First, notice that maximizing the number of pairs of adjacent edges is equivalent to minimizing the number of pairs of disjoint edges.
Then one can generalize the idea of a disjoint pair of edges to a matching (a set of pairwise disjoint edges).
The Keough-Radcliffe Theorem states that among all graph in $\mathscr{G}(n,m)$, 
the lexicographic or colexicographic graph minimizes the number of matchings, 
and also minimizes the number of $k$-matchings.
It would be interesting to generalize this theorem with the ideas of Theorem \ref{triangleProof}.
We give some ideas along these lines in the final section of the paper. 

\subsection{Generalized Edge-Isoperimetric Functions}

We now turn our attention to Lindsay's theorem and generalized edge-isoperimetric functions.
\begin{dfn}[Cartesian Graph Products]\label{graphProducts}
	For graphs $G$ and $H$ their \textit{Cartesian product} is
	a graph $G\square H$ defined as follows:
	\begin{align*}
		V_{G\square H} &= V_G \times V_H\\
		E_{G\square H} &= \{ ((v_G,v_H), (u_G,u_H) ) \bigm |
		v_G=u_G \text{ and } (v_H,u_H)\in E_H, \text{ or } 
		v_H=u_H \text{ and } (v_G,u_G)\in E_G\}.
	\end{align*}
	Let $G^d=G\square \cdots \square G$ ($d$ times). 
	Note that $G^0$ is a simple graph with one vertex.
\end{dfn}

For any $n\in \N$ let $K_n=(V,E)$ be the complete graph on vertex set $V=[n]_0$.
With this setup, the vertex set of an arbitrary product of complete graphs is a multiset lattice.
Thus, we will talk about initial segments of domination orders.
We are now ready to state Lindsay's Theorem.

\begin{thm}[Lindsay \cite{LINDSEYJH1964AoNt}]\label{cliqueProduct}
	If $n_1\leq n_1\leq \cdots \leq n_d$ then every initial segment of $\mathcal{L}$ is optimal in $G=K_{n_1} \square K_{n_2}\square\cdots \square K_{n_d}$ for $I_G$ and $\Theta_G$.
\end{thm}

It turns out that optimal sets in the product of complete graphs are found among downsets.
Furthermore, the edge-isoperimetric problem reduces to a maximum weight downset problem.
This is a general property of graphs with nested solutions and we state it as such.

\begin{dfn}[Nested Solutions and $\delta$-sequences]\label{nestedSolutonsInduced}
	Let $G=(V,E)$ be a graph.
	We say that $G$ has \textit{nested solutions} if there is a sequence of subsets $A_0\subseteq A_1\subseteq \cdots \subseteq A_{|V|}$,
	such that for each $i$ we have that $|A_i|=i$ and $A_i$ is optimal under $I_G$.
	
	Equivalently, if a graph $G$ has nested solutions then its vertex set can be totally ordered by some order $\mathcal{O}$,
	such that each initial segment of $\mathcal{O}$ is an optimal set under $I_G$.
	We call such an order $\mathcal{O}$ an \textit{optimal order}.
	Hence, vertex set of a Cartesian product of graphs is then just a product of totally ordered sets.
	Thus, the vertex set of a Cartesian product of graphs is isomorphic to a multiset lattice by Proposition \ref{decompositionOfMultisetLatices}.
	
	Then we define the $\delta$\textit{-sequence} of $G$ (see \cite{AHLSWEDE1997355,AHLSWEDE1997479, BezrukovSergei2018Nifo, BezrukovS.L.2000AEPf, BezrukovSergeiL.1999Oaei, BezrukovSergeiL.2007Opoh, BezrukovSergeiL.2003Epfc, BezrukovSergeiL.2004Anat, BonnetEdouard2016AnoE}) such that for any $i\in [|V|]$ we have
	\begin{align*}
		\delta_G(i) = |I_G(\mathcal{O}^{-1}[i])| - |I_G(\mathcal{O}^{-1}[i-1])|.
	\end{align*}
	Note that any optimal order gives the same $\delta$-sequence.
\end{dfn}

For example, $K_n$ has nested solutions and 
\begin{align*}
	\delta_{K_n} = (0,1,2,3,4,\dots, n-1).
\end{align*}
Of course, one can define such nested solutions and $\delta$-sequences for $\Theta_G$ in a similar way to $I_G$.

The nested solutions make the vertex set into a totally ordered set.
The ideas of the following lemma are presented by Harper in \cite{HarperBook}. Harper credits Bezrukov with this insight.

\begin{lem}
	Suppose that we have graphs $G_i = (V_i,E_i)$ for $i\in [d]$, that have nested solutions.
	Let $G_1\square \cdots \square G_d = (V,E)$ and $\sigma:V \rightarrow \mathscr{M}=\mathscr{M}_{[d]}(|V_1|,\dots, |V_d|)$ be the decomposition isomorphism in Proposition \ref{decompositionOfMultisetLatices}.
	For any $A\subseteq \mathscr{M}$ there is a downset $B\subseteq \mathscr{M}$ such that:
	\begin{enumerate}
		\item $|B|=|A|$.
		\item $ |I_G(\sigma^{-1}(B))|\geq |I_G(\sigma^{-1}(A))|$.
		\item $|I_G(\sigma^{-1}(B))| =  \sum_{(x_1,\dots, x_d)\in B} \sum_{i=1}^d \delta_{G_{i}}(x_i)$.
	\end{enumerate}
\end{lem}

Given this lemma, we can see that solving the edge-isoperimetric problem on the product of complete graphs,
is equivalent to solving the maximum weight downset problem on the corresponding multiset lattice,
where we are working with the standard weight function.
Thus, all the results from sections \ref{rectangles} and \ref{rectanglesExact} apply to the product of two complete graphs.
We get results about all optimal downsets, the uniqueness of the nested solutions, and the comparison between the lexicographic and colexicographic orders.

Ahlswede and Cai in \cite{AHLSWEDE1997479} proved that $G\square \cdots \square G$ has lexicographical nested solutions if $G \square G$ has lexicographical nested solutions, when $|V_G|\geq 3$.
To be more precise they proved such a result for generalized edge-isoperimetric functions.
Harper in \cite{HarperBook} proves this theorem for an arbitrary product of graphs, and calls this result the Ahlswede-Cai local-global principle.
%Given the local-global principle it is very little extra work to prove Lindsay's edge-isoperimetric inequality in any dimension.

\begin{thm}[Harper \cite{HarperBook}]\label{HarperLgp}
	Suppose that we have graphs $G_i = (V_i,E_i)$ for $i\in [d]$, that have nested solutions.
	If for all $i<j$ the product $G_i\square G_j$ has lexicographic nested solutions then $G_1\square \cdots \square G_d$ has lexicographic nested solutions. 
\end{thm}

Combining Harper's version of the local-global principle with our results on rectangles gives us Lindsay's Theorem.
We can actually prove even more along these lines by considering the original local-global principle by Ahlswede and Cai.
For this we first need to define generalized edge-isoperimetric functions.

\begin{dfn}[Generalized Edge-Isoperimetric Functions]
	Let $\varphi:2^{[n]_0} \rightarrow \R$ be a function.
	We are now going to define the \textit{$d$-th power of} $\varphi$, and denote it by $\varphi^d$.
	The range of this new function is $\R$ and the domain of it is $2^{[n]_0\times \cdots \times [n]_0}$, where we have a $d$-fold Cartesian product.
	Notice that $[n]_0\times \cdots \times [n]_0$ is the multiset lattice $\mathscr{M}=\mathscr{M}_{[d]}(n,\dots, n)$.
	For any $A\subseteq \mathscr{M}_{[d]}(n,\dots, n)$ we define $\varphi^d(A)$ to be the sum of $\varphi$ applied to all $1$-dimensional subproducts intersected with $A$.
	To be precise, let
	\begin{align*}
		\varphi^d(A) = \sum_{i=1}^d \sum_{x\in \mathscr{M}_{[d]\setminus \{i\}}} \varphi(\mathscr{M}_{\{i\}} (x)\cap A).
	\end{align*}
	The function $\phi^d$ is called a \textit{generalized edge-isoperimetric function}.
\end{dfn}

The most important examples of generalized edge-isoperimetric functions are $I_G^d$ and $\Theta_G^d$. 

\begin{dfn}[Push-Down Functions]
	Let $\varphi:2^{[n]_0} \rightarrow \R$ be a function.
	We say that $\varphi$ is a \textit{push-down} function if it satisfies the following properties:
	\begin{enumerate}
		\item (nestedness/nested solutions) for all $k\in [n]_0$ and all $A\subseteq [n]_0$ we have
		\begin{align*}
			\varphi(A) \leq \varphi([k]).
		\end{align*}
		\item (submodularity) for all $A,B\subseteq [n]_0$ we have
		\begin{align*}
			\varphi(A) + \varphi(B) \leq \varphi(A\cup B) + \varphi(A\cap B).
		\end{align*}
		\item $\varphi(\emptyset) = 0$.
	\end{enumerate}
\end{dfn}

First, note that the third property of push-down functions is just for convenience, 
since we can always define $\varphi'$ such that $\varphi'(A) = \varphi(A) - \varphi(\emptyset)$.
Second, notice that $I_G$ and $-\Theta_G$ satisfy the second and third properties for push-down functions for any graph $G$.
If the graph $G$ has nested solutions then they satisfy the first property as well.

\begin{prb}[Maximizing Push-Down Functions]\label{submodularProblem}
	Given a push-down function $\varphi:2^{[n]_0} \rightarrow \R$, $d\geq 1$, and $m\in [n+1]_0$, find a set $A\subseteq [n]_0$ such that $|A|=m$ and  
	\begin{align*}
		\varphi^d(A) = \max_{\substack{S\subseteq [n]_0\\ |S|=m}} \varphi^d(S).
	\end{align*}
	We call such a set $A$ \textit{optimal for} $\varphi^d$.
\end{prb}

\begin{dfn}[$\delta$-sequences]
	For any push-down function $\varphi: 2^{[n]_0} \rightarrow \R$ we define its $\delta$-sequence such that for any $i\in [n]$
	\begin{align*}
		\delta_{\varphi}(i) = \varphi([i]_0) - \varphi([i-1]_0).
	\end{align*}
\end{dfn}

\begin{lem}[Ahlswede-Cai \cite{AHLSWEDE1997355}]
	Suppose that $\varphi: 2^{[n]_0} \rightarrow \R$ is a push-down function and $d\in \N$ with $d\geq 1$.
	If $A\subseteq \mathscr{M}=\mathscr{M}_{[d]}(n,\dots, n)$ then:
	\begin{enumerate}
		\item there exists a downset $B\subseteq \mathscr{M}$ such that $\varphi^d(A) \leq \varphi^d(B)$.
		\item $\varphi^d(B) = \sum_{(x_1,\dots, x_d)\in B} \sum_{i=1}^d \delta_{\varphi}(x_i)$.
	\end{enumerate}
\end{lem}

Thus, Problem \ref{submodularProblem} is a special case of the maximum weight downset problem.
In particular, our results for rectangles apply to any push-down function $\varphi:2^{[n]_0}\rightarrow \R$ and its second power $\varphi^2$,
whenever $\delta_\varphi = (0, c, 2c,3c,\dots, (n-1)c)$ for some $c\in \R$.

\begin{thm}[Ahlswede-Cai \cite{AHLSWEDE1997479}]\label{lgbAC}
	Suppose that $\varphi: 2^{[n]_0} \rightarrow \R$ is a push-down function with $n\geq 3$.
	If all initial segments of $\mathcal{L}$ are optimal under $\varphi^2$ then all initial segments of $\mathcal{L}$ are optimal under $\varphi^d$ for any $d\geq 1$. 
\end{thm}

Thus, we get the following corollary by combining Theorem \ref{lgbAC} and Corollary \ref{nestedSolutionsRectangles}.
\begin{cor}\label{lastCor}
	Suppose that $\varphi: 2^{[n]_0} \rightarrow \R$ is a push-down function with $n\geq 3$, and $\delta_\varphi = (0,c,2c,\dots, (n-1)c)$ for some $c\in \R$.
	Then all initial segments of $\mathcal{L}$ are optimal under $\varphi^d$ for any $d\geq 1$.
\end{cor}

Corollary \ref{nestedSolutionsRectangles} can tell us something even stronger.
Discrete isoperimetric problems typically involve maximizing (respectively minimizing) some submodular (respectively supermodular, the inequality flips) function $\varphi:2^{[n]_0}\rightarrow \R$.
If $\varphi$ has nested solutions then a lot of the time we can reduce to maximizing over downsets.
Then to solve the problem for $\varphi^d$, pushing-down compression is used with respect to some order $\mathcal{O}$.
For an order $\mathcal{O}$ on $\mathscr{M} = \mathscr{M}_{[d]}(n,\dots, n)$, 
set of coordinates $S\subseteq [d]$, 
and a set $A\subseteq \mathscr{M}$ we define the pushing-down compression of $A$ with respect to $S$ by
\begin{align*}
	C_S(A) = \bigcup_{x\in \mathscr{M}_{[d]\setminus S}} \mathscr{M}^{-1}_S(x)[|\mathscr{M}_S(x)\cap A|].
\end{align*}
Theorems \ref{HarperLgp} and \ref{lgbAC} are proved by applying pushing-down operation with respect to $\mathcal{O} = \mathcal{L}$ 
until $C_S(A)=A$ (such a  set is called compressed) for any $S\subseteq [d]$.
The typical property that occurs is $\varphi^d(A) \leq \varphi^d(C_S(A))$.
The key component to using pushing-down compression is the order $\mathcal{O}$.
In particular, we can inductively solve the problem for $\varphi^d$ if we know that all initial segments of $\mathcal{O}$ are optimal for $\varphi^2$.
This is where Corollary \ref{nestedSolutionsRectangles} gives us uniqueness.
If one was to apply such a general technique of pushing-down compression when $\delta_{\varphi}=(0,c,2c,\dots, (n-1)c)$, 
then we have at most two options for $\mathcal{O}$, namely $\mathcal{O} = \mathcal{L}$ and $\mathcal{O}=\mathcal{C}$.

We should mention that Corollary \ref{lastCor} has been discovered before as a special case of Clements-Lindström Theorem \cite{ClementsG.F.1969Agoa}.
Both Engel \cite{engel_1997} and Harper \cite{HarperBook} cover it in the chapter (Chapter 8 for both) on Macaulay posets.
Harper covers it as a special case of the vertex isoperimetric problem.
Engel covers it from the shadow minimization perspective and provides a detailed history of relation between shadow minimization problems and maximum weight ideal problems.
The uniqueness of the orders $\mathcal{L}$ and $\mathcal{C}$ has not been discussed before.

%auto-ignore
\section{Final Remarks and Future Directions}\label{remarks}

As mentioned in section \ref{appliations}, 
generalizing the Keough-Radcliffe matching theorem along the lines of Theorem \ref{triangleProof} would be interesting to see.
Do symmetrization and reflect-push method increase the number of matching?
One can ask an even more general question.
Let $J(n,2)=(V,E)$.
What properties does a function $\varphi: 2^V \rightarrow \R$ need to satisfy such that for any set $A\subseteq V$,
symmetrizations of $A$ and reflect-push methods on $A$,
that produce a new set $B$,
satisfy $\varphi(A)\leq \varphi(B)$ (or $\varphi(A) \geq (B)$)?

It would be exciting to see a solution to the edge-isoperimetric problem in $I(n,3,2)$.
This question was first asked 45 years ago by Ahlswede and Katona \cite{Ahlswede1978GraphsWM}.
As mentioned in the introduction there is a lot of interest in this area.
However, this problem is extremely difficult as there have been many attempts over such a long period that only give partial results.
Notice that rectangles and right triangles are subposets of the vertex set on $I(n,3,2)$.
So, one can get partial results about the edge-isoperimetric problem from our two dimensional results.
We hope to study more types of subposets of the vertex set on $I(n,3,2)$, as subposet ideas were very useful in proving Theorem \ref{triangleProof}.

The first type of poset is right trapezoids.
These are generalization of right triangles,
the posets that keep all the increasing or decreasing sequences of $\mathscr{M}_{[2]}(\ell_1,\ell_2)$.
Another interesting case to look as is posets of the form $\mathscr{R}_{[2]}(\ell_1) \times \mathscr{M}_{[1]}(\ell_2)$.
We call these types of posets prisms.
Solving the maximum weight ideal problem on these posets could give ideas on how to approach the edge-isoperimetric problem in $I(n,3,2)$.

We predict that a very similar results hold for right trapezoids, as the ones for right triangles.
The case for the prisms is much more interesting and difficult.
If one follows the diagonal ($(1,2,3), (2,3,4), (3,4,5),\dots $) of $I(n,3,2)$ and stops,
then the poset that includes the all the vertices with higher $z$-coordinate can be partitioned into two prisms.
If all optimal downsets in prisms are found then similar techniques to the proof of Theorem \ref{triangleProof} could work to handle the edge-isoperimetric problem on $I(n,3,2)$.

\bibliographystyle{acm}
\bibliography{./reflect_push}
\end{document}